\documentclass[11pt,reqno]{amsart}
\usepackage{amsmath}
\usepackage{amsfonts}

\numberwithin{equation}{section}
\usepackage{mathrsfs}
\usepackage{times}
\usepackage{amsmath,amsfonts,amstext,amssymb,amsbsy,amsopn,amsthm,eucal}
\usepackage{txfonts}
\usepackage{dsfont}
\usepackage{graphicx}   
\usepackage{hyperref}
\usepackage{accents}
\usepackage{enumerate}
\usepackage{xcolor}

\usepackage{verbatim}

\newtheorem{theorem}{Theorem}[section]

\newtheorem{lemma}[theorem]{Lemma}

\theoremstyle{definition}
\newtheorem{definition}[theorem]{Definition}
\theoremstyle{remark}
\newtheorem{remark}{Remark}[section]
\theoremstyle{remark}

\theoremstyle{remark}

\theoremstyle{remark}
\theoremstyle{remark}
\DeclareMathOperator*\lowlim{\underline{lim}}
\DeclareMathOperator*\uplim{\overline{lim}}
\def\comp{\ensuremath\mathop{\scalebox{.6}{$\circ$}}}

\begin{document}

\title{The Penrose Inequality for metrics with singular sets}
\date{\today}

\author{Huaiyu Zhang}
\email{zhymath@outlook.com}
\address{School of Mathematics and Statistics, Nanjing University of Science and Technology}

\begin{abstract}
We study the Penrose inequality and its rigidity for metrics with singular sets. Our result could be viewed as a complement of Theorem 1.1 of Lu and Miao ( J. Funct. Anal. 281, 2021) and Theorem 1.2 of Shi, Wang and Yu (Math. Z. 291, 2019), in which they assume the singular set is a hypersurface and assume an additional condition on the mean curvature. As a complement, this paper study the case of singular set of dimensional less than $n-1$, without any additional conditions.

\vspace{10pt}

\vspace{10pt}
\end{abstract}

\maketitle
\section{itroduction}
The positive mass theorem and the Penrose inequality (in Riemannin case) are two of the most celebrated theorems in Riemannian geometry with nonnegative scalar curvature. The positive mass theorem compares the ADM mass of an asymptotically flat manifold with that of the Euclidean space, which was proved by Schoen-Yau in \cite{SY79,ScYa2} and by Witten in \cite{Witten},  while the Riemannian Penrose inequality compares the ADM mass of an asymptotically flat manifold with that of the spatial Schwarzschild manifold, which was proved by Huisken-Ilmanen in \cite{HI01}, and by Bray in \cite{Br}. Later it was generalized to manifolds of dimensional less than $8$ by Bray-Lee in \cite{BL09}. See also the book \cite{Le19} for a good introduction to this two classical theorems.

After that, many researchers have been considering whether these theorems still hold for metrics with singular sets while the scalar curvature is only assumed to be nonnegative pointwisely away from the singular set. On the one hand, Miao, Shi-Tam and McFeron-Sz\'ekelyhidi proved a positive mass theorem for Lipschitz metrics with singularity along a hypersurface in  \cite{Miao2002},  \cite{ST02} and \cite{McSz}, independently by three different approaches. In this case, a condition on the mean curvature along the hypersurface is assumed, and there exists counterexamples without this condition. Later, in \cite{Le13}, Lee proved a positive mass theorem for Lipschitz metrics with small singular sets and proposed a conjecture, which was confirmed by Lee-LeFloch in \cite{Le13a} for the spin case and by Jiang, Sheng and the author in \cite{JSZ} for the general case. For more results about the positive mass theorem for metrics admitting small singular set, see \cite{Le13a}, \cite{LiMa}, \cite{ShiTam}, \cite{CL22}, \cite{LT21}, etc.

 On the other hand, McCormick-Miao proved the inequality part of the (Riemannian) Penrose inequality for Lipschitz metrics with singularity along a hypersurface in \cite{StMi}. Later, the rigidity part was proved by Shi-Wang-Yu in \cite{SWY19} and by Lu-Miao in \cite{LuMi} independently by different approaches. Similarly, an indispensable condition on the mean curvature along the hypersurface is assumed again in these results, otherwise counterexpamles would appear. It is natural to ask  a (Riemannian) Penrose inequality for metrics with small singular sets, without such an additional condition.

In this paper, we consider this problem and our main theorem is:

\begin{theorem}\label{mthm}
Let $M^n(3\le n\le 7)$ be a smooth manifold with a boundary $\Sigma_H$  and $g$ be a Lipschitz metric on $M$ which is smooth away from a compact singular set $S\subset M\setminus \Sigma_H$. Suppose $\Sigma_H$  is minimal and strictly outer-minimizing, $g$ is asymptotically flat and has nonnegative scalar curvature away from $S$. If the $(n-1)$-dimensional lower Minkowski content of $S$ is zero. Then we have 
\begin{align*}
m(g)\ge \frac{1}{2}\left(\frac{|\Sigma_H|}{\omega_{n-1}}\right)^\frac{n-2}{n-1},
\end{align*}
where $m(g)$ is the ADM mass of $g$, $|\Sigma_H|$ is the area of $\Sigma_H$, $\omega_{n-1}$ is the area of the standard round $n-1$-dimensional sphere.

Moreover, if $m(g)= \frac{1}{2}\left(\frac{|\Sigma_H|}{\omega_{n-1}}\right)^\frac{n-2}{n-1}$, then there exists a $C^{1,\alpha}$ diffeomorphism $\Phi:\mathbb{M}_m \to M$, such that $g_m=\Phi^* g$, where $\alpha\in(0,1)$ and $(\mathbb{M}_m,g_m)$ is the standard spatial Schwarzschild manifold with mass $m$.
\end{theorem}

\begin{remark}
The condition that $S$ is of vanishing $(n-1)$-dimensional lower Minkowski content is optimal in some sense. For the definition of the lower Minkowski content, see Definition \ref{Mc}. For now, recall for well behaved sets the Minkowski content equals Hausdorff measure. Moreover, if $S$ is a submanifold, then $S$ is of vanishing $(n-1)$-dimensional lower Minkowski content if and only if $S$ is of dimensional less than $n-1$. 

In fact, if $S$ is of dimensional $n-1$, one can construct counterexamples for which the ADM mass is negative. See also \cite{StMi}, \cite{SWY19} and \cite{LuMi}, in which $S$ is assumed to be a hypersurface, but an additional condition on its mean curvature is required. 

\cite{Le13a}, \cite{ShiTam}, \cite{LT21}, \cite{CL22}, etc. also use the Minkowski content to give  dimensional condition for the singular set.
\end{remark}
 
\begin{remark}
In \cite{Le13}, Lee proposed a conjecture that the positive mass theorem holds for Lipschitz metrics with  singular sets of dimensional less than $n-1$, which was proved by Lee-LeFloch in \cite{Le13a} for the spin case and by Jiang, Sheng and the author in \cite{JSZ} for the general case. Theorem \ref{mthm} affirms the Penrose inequality version of this conjecture.
\end{remark}

In \cite{LuMi}'s proof using Bray's conformal metrics flow method, a key step is to solve the conformal factor in their equation (2.3):
\begin{align}
\tilde g=\left(\frac{\bar \varphi+1}{2}\right)^\frac{4}{n-2} g,
\end{align}
where $\bar \varphi$ is weakly $g$-harmonic.

By proving  that $\tilde g$ is Euclidean and the conformal factor must be the function $\left(1+\frac{m}{2|x|^{n-2}}\right)^\frac{4}{2-n}$, they showed $g$ must be the Schwarzschild metric. In Bray's proof in \cite{Br}, he proved that $\tilde g$ is Euclidean by using the classical positive mass theorem. And his $g$ was assumed to be smooth, thus the equation  was solved quiet straightforwardly. Along the same idea, in \cite{LuMi}'s proof, Lu-Miao found a smart approach to  improve the results in \cite{Miao2002},  \cite{ST02} and \cite{McSz}, which enabled them to say that $\tilde g$ is Euclidean. They also proved that $g$ is smooth, then the equation could be solved. However, Lu-Miao's approach depends on an analysis of the second fundamental form of the corners and it seems that this approach does not work in our case. 

To overcome this difficulty, we prove a new singular positive mass theorem, which enables us to say that $\tilde g$ could be pullback to be Euclidean by a $C^{1,\alpha}$ diffeomorphism, for some $\alpha\in(0,1)$. Then we are able to solve the conformal factor under this pullback.

\textbf{Organization:} In Section 2, we prove a new singular positive mass theorem (Theorem \ref{PMT}), since we need a singular positive mass theorem delicate enough as  Theorem \ref{PMT} to prove our main theorem. In Section 3, we combine the approach in \cite[Proposition 3.1]{StMi} and our results in Section 2  to prove the inequality part of Theorem \ref{mthm}. In Section 4, we prove our main theorem, of which the key step is to prove Bray's mass-capacity inequality for metrics with singular sets.

\section{A singular positive mass theorem}
In this section, we prove a singualar positve mass theorem, which will play a role in the next two sections.

\begin{theorem}\label{PMT}
Let $M^n(n\ge3)$ be a smooth manifold with a metric $g\in C^\infty(M\setminus S)$, where $S$ is a subset of $M$ and $S=\cup_{k=0}^{k_0} S^k$, each $S^k(0\le k\le {k_0})$ is compact and $S^k\cap S^l=\emptyset$ for $k\neq l$, such that 
\begin{enumerate}
\item $S^0$ is a hypersurface which bounds a domain $\Omega$, $g$ is Lipschitz in a neighborhood of $S^0$, and the mean curvature of $S^0$ in $(\bar\Omega,g|_\Omega)$ is not less than  it in $(M\setminus \Omega,g|_{M\setminus \Omega})$.

\item For each $1\le \iota\le k_0$,  $g$ is $W^{1,p_\iota}$ in a neighborhood of $S^\iota$ for some $p_\iota \in (n,+\infty]$, and the $(n-1-n/{p_\iota})$-dimensional lower Minkowski content of $S^\iota$ is zero.
\end{enumerate}

Suppose $g$ is complete and asymptotically flat, then its ADM mass is nonnegative. Moreoever, if the mass is zero, then there exists a $C^{1,\alpha}$ diffeomorphism $\Phi:\mathbb{R}^n \to M$, such that $g_{\text{Euc}}=\Phi^* g$, where $g_{\text{Euc}}$ is the standard Euclidean metric.
\end{theorem}

\begin{remark}
Particularly, in the condition (2), if $g$ is Lipschitz in a neighborhood of $S^\iota$, then $p_\iota=+\infty$, and $S^\iota$ is assumed to have $(n-1)$-dimensional lower Minkowski content vanishing. This dimensional condition of the singular sets is optimal in some sense, Since if $S^\iota$ is a hypersurface, without the mean curvature condition, then there is a counterexample for which the mass may be negative.
\end{remark}

\begin{remark}
Theorem \ref{PMT} can be considered as an extension of  \cite{Miao2002}, \cite{ST02}, \cite{McSz}, \cite{ShiTam} and \cite{Le13}. In fact, if $S=S^0$, then Theorem \ref{PMT} is reduced to the case in \cite{Miao2002}, \cite{ST02} and \cite{McSz}; if $S=\cup_{\iota=1}^m S^\iota$, then Theorem \ref{PMT} is reduced to the case in \cite{ShiTam} and \cite{Le13a} (with a refinement on the dimensional condition on the singular sets). However, the results in these references and their combinations are not delicate enough for our use in the next two sections. Thus we must prove Theorem \ref{PMT}.
\end{remark}

\begin{remark}
One could compare Theorem \ref{PMT} with \cite[Theorem 1.1]{JSZ}, which is a previous work by Jiang, Sheng and the author. In fact, though not been explicity claimed there, the argument in \cite{JSZ} actually gives the inequality part of Theorem \ref{PMT}. It is because one can check that $g$ has nonnegative scalar curvature in distribuional sense under the assumption of Theorem \ref{PMT} (by combining \cite[Proposition 5.1]{Le13a} and \cite[Lemma 2.7]{JSZ}). However, the argument in \cite{JSZ} does not give the rigidity part of Theorem \ref{PMT}. In this section, therefore, we will prove Theorem \ref{PMT} by using the Ricci flow method in \cite{McSz} and \cite{ShiTam}.
\end{remark}

Throughout this paper, we fix an asymptotically flat coordinate, and let $h$ be a smooth background metric  which is also asymptotically flat. We let $\tilde \nabla$ and $|\cdot|$ denote the Levi-Civita connection and the norm taken with respect to $h$ throughout this paper. We also let $B^h_r(x)$ denote the geodesic ball centered at $x$ with radius $r$ taken with respect to $h$. 

It is standard to define the Lebesgue space on a measurable subset $A$ for a tensor field $T$ as:
\begin{align}
\|T\|_{L^p(M)}:=\left(\int_A |T|^p d\mu_h\right)^\frac{1}{p},
\end{align}
where $d\mu_h$ is the volume form taken with respect to $h$.

Then the local Sobolev space $W^{k,p}_{\text{loc}}$, for the tensor fields of type $(s,t)$  which have weak derivative up to $k$th order on $M$, is defined as:
\begin{align}
W^{k,p}_{\text{loc}}:=\{T\in T^s_t(M): \exists A>0, {\text{such that}} \ \ \|\tilde\nabla ^l T\|_{L^p(B^h_1(x))}<A, \forall 1\le l\le k, x\in M\}
\end{align}

This definition does not depend upon the choice of the background metric $h$.

Let $(M,g)$ be assumed as in Theorem \ref{PMT} and denote $p=\inf_{1\le \iota\le k_0} p_\iota$. Then we have $g\in W^{1,p}_{\text{loc}}$ and $p\in(n,+\infty]$. 

\textbf{Organization of this section:} In Subsection 2.1, we construct smooth metrics approximating the singular metric such that the negative part of their scalar curvature converge to zero in $L^1$ sense.
 In Subsection 2.2, we recall the definitions of the Ricci flow and the $h$-flow, and we obtain some estimates for them. In Subsection 2.3, we check  scalar curvature along the flow. In Subsection 2.4, we prove Theorem \ref{PMT}.

\subsection{Approximation of the singular metric}

In this subsection we prove the following lemma:

\begin{lemma}\label{thm3.1}
Let $(M,g)$ be assumed as in Theorem \ref{PMT}, then there exists a family of smooth metrics $g_\epsilon$ such that\\
 (1) $g_\epsilon=g$ on $M\setminus S_{2\epsilon}$, wehre $S_{2\epsilon}$ is the $2\epsilon$-neighborhood of $S$.\\
 (2) $g_\epsilon\to g$  on $M$ in $C^0$ norm as $\epsilon\to 0^+$ .\\
 (3) there exists an $A=A(g)>0$ which does not depend on $\epsilon$ and $x$, such that $\int_{B^h_1(x)}|\tilde\nabla g_\epsilon|^pd\mu_h<A$, where $p=\inf_{1\le \iota\le k_0} p_\iota$ and $B^h_1(x)$ is the geodesic ball of $h$.\\
 (4) Moreover, if $g$ has nonnegative scalar curvature, then \[\lowlim_{\epsilon\to0^+}\int_M (R_{g_\epsilon})_-d\mu_h= 0,\] where $(R_{g_\epsilon})_-$ is the negative part of the scalar curvature of $g_\epsilon$.
\end{lemma}
\begin{remark}
The construction is actually a combination of the constructions in \cite{Miao2002} and in \cite{Le13}. In fact, 
Lee  has proved $\lim_{\epsilon\to0^+}\int_M (R_{g_\epsilon})_-^{\frac{n}{2}}d\mu_h = 0$ in \cite{Le13} while assuming a stronger condition on the Minkowski content. However, we will check that $\lim_{\epsilon\to0^+}\int_M (R_{g_\epsilon})_-d\mu_h = 0$ in our case, which is good enough for our use.
\end{remark}

Let us explain the concept of the lower Minkowski content here:
\begin{definition}\label{Mc}
For a subset $S\in M$, given a $C^0$  metric $g$, the $m$-dimensional lower Minkowski content of $S$ is 
\begin{align*}
\lowlim_{\epsilon\to 0^+}\frac{\mu_g(S_\epsilon)}{\omega_{n-m}\epsilon^{n-m}},
\end{align*}
where $\mu_g$ is the volume form of $g$, $S_\epsilon$ is the $\epsilon$-neighborhood of $S$ and $\omega_{n-m}$ is the volume of the unit ball in $\mathbb{R}^{n-m}$.
\end{definition}

Firstly, we consider a particular case in which the singular set is of only one type. We have the following lemma:
\begin{lemma}
Let $(M,g)$ be assumed as in Theorem \ref{PMT}. Suppose that $S=S^1$, in other words, $S^0=\emptyset$ and $k_0=1$, $p=p_1$, then we con construct a family of metrics satisfying the conclusions of Theorem \ref{PMT}
\end{lemma}

\begin{proof}
Since $g$ is asymptotically flat and continuous, we can choose a finite atlas $\{U_k\}_{k=1}^N$for $M$, such that $g$ is uniformly equivalent to the Euclidean metric of each chart, i.e. there exists a constant $C$, such that on each chart we have 
\[C^{-1}g_{\text{Euc}}\leq g\leq Cg_{\text{Euc}},\] 
where $g_{\text{Euc}}$ is the Euclidean metric on the chart, (here and below $C=C(g)$ will denote a positive constant which does not depend on $\epsilon$ and varies from line to line). Let $\psi_k$ be the partition of unity subordinate to the atlas, and denote $U'_k:=\text{supp}\psi_k$.

By \cite[Lemma 3.1]{Le13}, on each chart $U_k$ there exists a nonnegative smooth function $\sigma$, such that $\sigma=\epsilon$ on $S_\epsilon$,  $\sigma=0$ on $M\setminus S_{2\epsilon}$, and $0\leq \sigma\leq \epsilon$ on $S_{2\epsilon}\backslash S_\epsilon$, $|\tilde\nabla\sigma|\leq C$ and $|\tilde\nabla^2\sigma|\leq C/\epsilon$. We denote $\phi$ as the standard mollifier defined on $\mathbb{R}^n$ that satisfies $\phi\in C_0^{\infty}(B_1^{\text{Euc}})$, $0\leq \phi\leq 1$, and $\int_{\mathbb{R}^n}\phi(x)dx=1$, where $B_1^{\text{Euc}}$ denotes the  unit ball in the Euclidean space, and we denote $\phi_\epsilon(x)=x^{-n}\phi(x/\epsilon)$.

Then we firstly define $g_\epsilon^k$ on each chart $U'_k$ by \[(g_\epsilon^k)_{ij}(x):=\int_{\mathbb{R}^n}g_{ij}(x-\sigma(x)y)\phi(y)dy,\] 
and then define $g_\epsilon$ on the whole $M$ by $g_\epsilon:=\Sigma_{k=1}^N \psi_k g_\epsilon^k$.

We now check propositions (1) to (4). Let $f$ denote any component $g_{ij}^k$, and $f_\epsilon$ denote the component $(g_\epsilon^k)_{ij}$.

For (1), since $\sigma=0$ on $M\setminus S_{2\epsilon}$, we have $g_\epsilon^k=g$ on $M\setminus S_{2\epsilon}$. Thus (1) holds.

For (2), on each chart $U_k$, by definition we have 
\begin{align*}
|f_\epsilon(x)-f(x)|&\leq \int_{B_1^{\text{Euc}}} \left|f(x-\sigma (x)y)-f(x)\right|\phi(y)dy.
\end{align*}
By Morrey's embedding, we have that $f$ is $C^0$ on $M$. Since $\phi\in C_0^{\infty}(B_1^{\text{Euc}})$,  and $0\le \sigma\le \epsilon$ on $M$, we have that $g_\epsilon^k\to g$  on $U_k$ in $C^0$ norm as $\epsilon\to 0^+$, thus (2) holds.

For (3), note that

\begin{align*}
|\partial f_\epsilon(x)|&=\left|\int_{B_1^{\text{Euc}}} \partial f(x-\sigma(x)y)(\partial x -\partial \sigma(x)y)\phi(y)dy\right|\\
&\leq C\int_{B_1^{\text{Euc}}} |\partial f|(x-\sigma(x)y)\phi(y)dy.
\end{align*}
 Take an arbitrary measurable set $A'\subset M$ and let $A=A'+\epsilon B_1^{\text{Euc}}$. By Minkowski inequality, we have
\begin{align*}
\|\partial f_\epsilon\|_{L^p(A'\cap U'_k)}&\leq C\left(\int_{A'\cap U'_k}\left(\int_{B_1^{\text{Euc}}(0)}|\partial f|(x-\sigma(x)y)\phi(y)dy\right)^pdx\right)^{\frac{1}{p}}\\
&\leq C\int_{B_1^{\text{Euc}}(0)}\left(\int_{A'\cap U'_k}|\partial f|^p(x-\sigma(x)y)dx\right)^{\frac{1}{p}}\phi(y)dy
\end{align*}

For $\epsilon$ small enough such that $U'_k+\epsilon B_1^{\text{Euc}} \subset U_k$, we have
\begin{align*}
\|\partial f_\epsilon\|_{L^p(A'\cap U'_k)}&\leq C\int_{B^h_1(0)}\|\partial f\|_{L^p(A\cap U_k)}\phi(y)dy\\
&=C\|\partial f\|_{L^p(A\cap U_k)},
\end{align*}

Therefore we have
\begin{align*}
\|\partial (\psi_k f_\epsilon)\|_{L^p(A'\cap U'_k)}
&\leq \|\partial \psi_k\|_{L^\infty(M} \| f_\epsilon\|_{L^p(A'\cap U'_k)}+\|\psi_k\|_{L^\infty(M} \| \partial f_\epsilon\|_{L^p(A'\cap U'_k)}\\
&\le C(\| f\|_{L^p(A\cap U_k)}+\|\partial f\|_{L^p(A\cap U_k)}),
\end{align*}
 thus (3) holds.

For (4), if $x\in U_k\backslash S_\epsilon$, we have
\begin{align*}
|\partial f_\epsilon(x)|&=\bigg|\int\partial f(x-\sigma(x)y)(\partial x -\partial \sigma(x)y)\phi(y)dy\bigg|\\&\leq\bigg|\int C(\|\partial x\|+\|\partial \sigma\|)\phi(y)dy\bigg|\\&\leq C,
\end{align*}
\begin{align*}
|\partial^2f_\epsilon(x)|=&\bigg|\int[\partial^2f(x-\sigma(x)y)(\partial x-\partial\sigma(y))^2\\&+\partial f(x-\sigma(x)y)(\partial^2x-\partial^2\sigma(x)y)\phi(y)]dy\bigg|\\&\leq C/\epsilon.
\end{align*}
Thus we have  $|\partial g_\epsilon|\le C$ and $|\partial^2 g_\epsilon|\le C/\epsilon$ on $U_k\setminus S_\epsilon$. Therefore, since the scalar curvature of metric $g$ is a contraction of $\partial^2 g+g^{-1}*g^{-1}*\partial g*\partial g$, as the same as in \cite{Le13}, we have
\begin{align}\label{1R1}
\lim_{\epsilon\to 0^+}\int_{S_{2\epsilon}\setminus S_\epsilon}|R_{g_\epsilon}|d\mu_h =0,
\end{align}
where $R_{g_\epsilon}$ is the scalar curvature of $g_\epsilon$.

If $x\in S_\epsilon$, then we change the variable in the integration and get $f_\epsilon(x)=\int_{B^h_1(0)} f(x-y)\phi_\epsilon(y)dy$. Thus we have 
\begin{align*}
|\partial^2 f_\epsilon(x)|&=\bigg|\partial ^2\int_{B^h_1(0)} f(x-y)\phi_\epsilon(y)dy\bigg|\\
&=\bigg|\partial \int_{B^h_1(0)}\partial f(x-y)\phi_\epsilon(y)dy\bigg|\\
&=\bigg|\partial \int_{B^h_\epsilon(x)}\partial f(y)\phi_\epsilon(x-y)dy\bigg|\\
&=\bigg|\int_{B^h_\epsilon(x)} \partial f(y)\epsilon^{-n-1}\partial \phi(\frac{x-y}{\epsilon})dy\bigg|\\
&\leq \|\partial f\|_{L^p(S_{2\epsilon}\cap U_k)}\bigg(\int_{B^h_\epsilon(x)} \big|\epsilon^{-n-1}\partial \phi(\frac{x-y}{\epsilon})\big |^{p*}dy\bigg)^{\frac{1}{p*}}\\
&\leq C\big((\epsilon^{-n-1})^{p*}\epsilon^n\big)^{\frac{1}{p*}}\\
&=C\epsilon^{-1-n/p},
\end{align*}
where $p*$ is the number such that $\frac{1}{p}+\frac{1}{p*}=1$. Thus we have $|\partial^2 g_\epsilon|\le C\epsilon^{-1-n/p}$ on $S_\epsilon$.

Since the scalar curvature of metric $g$ is a contraction of $\partial^2 g+g^{-1}*g^{-1}*\partial g*\partial g$, we have
\begin{align}\label{1R2}
\int_{U_k}(R_{g_\epsilon})_-d\mu_h
&\le \int_{S_{\epsilon}\cap U_k}|R_{g_\epsilon}|d\mu_h+\int_{(S_{2\epsilon}\setminus S_\epsilon)\cap U_k}|R_{g_\epsilon}|d\mu_h\notag\\
&\leq C\int_{S_{\epsilon}\cap U_k}|\partial^2 g_\epsilon|dx+C\int_{S_{\epsilon}\cap U_k}|\partial g_\epsilon|^2dx+\int_{(S_{2\epsilon}\setminus S_\epsilon)\cap U_k}|R_{g_\epsilon}|d\mu_h\notag\\
&\leq C\mu_g(S_{\epsilon}\cap U_k)\epsilon^{(-1-n/p)}+C\|\partial g_\epsilon\|_{L^2(S_{\epsilon}\cap U_k)}^2+\int_{(S_{2\epsilon}\setminus S_\epsilon)\cap U_k}|R_{g_\epsilon}|d\mu_h\notag\\
&\leq C\frac{\mu_g(S_{\epsilon}\cap U_k)}{\epsilon^{(1+n/p)}}+C\left(\left(\mu_g(S_{\epsilon}\cap U_k)\right)^{\frac{1}{2}-\frac{1}{p}}\|\partial g_\epsilon\|_{L^2(S_{\epsilon}\cap U_k)}\right)^2+\int_{(S_{2\epsilon}\setminus S_\epsilon)\cap U_k}|R_{g_\epsilon}|d\mu_h.
\end{align}

Since the $(n-1-n/p)$-dimensional lower Minkowski content of $S$ is zero, we have
\begin{align}\label{1R3}
\lowlim_{\epsilon\to 0^+}\frac{\mu_g(S_{\epsilon}\cap U_k)}{\epsilon^{(1+n/p)}}=0.
\end{align}
And by (3), we have
\begin{align}\label{1R4}
\lowlim_{\epsilon\to 0^+}\left(\mu_g(S_{\epsilon}\cap U_k)\right)^{\frac{1}{2}-\frac{1}{p}}\|\partial g_\epsilon\|_{L^2(S_{\epsilon}\cap U_k)}=0
\end{align}

By \eqref{1R1}, \eqref{1R2}, \eqref{1R3} and \eqref{1R4}, we have
\begin{align}
\lowlim_{\epsilon\to 0^+}\int_{U_k}(R_{g_\epsilon})_-d\mu_h=0
\end{align}

 Recall that $g_\epsilon=g$ on $M\setminus S_{2\epsilon}$ and the scalar curvature of $g$ is nonnegative on $M\setminus S$, we have (4) holds, thus the lemma is proved.
\end{proof}

Note that the mollification and estimates in the proof above are  local, actually we  have proved the following:

\begin{lemma}\label{thm3.11}
Let $(M,g)$ be assumed as in Lemma \ref{PMT}, then for each $\iota=1,...,k_0$ there exists a family of metrics $g_\epsilon^\iota$ which is smooth on $(M\setminus S)\cup S^\iota$ such that\\
 (1) $g_\epsilon^\iota=g$ on $M\setminus S_{2\epsilon}^\iota$.\\
 (2) $g_\epsilon^\iota\to g$  on $M$ in $C^0$ norm as $\epsilon\to 0^+$ .\\
 (3) there exists an $A=A(g)>0$ which does not depend on $\epsilon$ and $x$, such that $\int_{B^h_1(x)}|\tilde\nabla g_\epsilon^\iota|^pd\mu_h<A$, where $p=\inf_{1\le \iota\le k_0} p_\iota$ and $B^h_1(x)$ is taken with respect to $h$.\\
 (4) Moreover, if $g$ has nonnegative scalar curvature, then \[\lowlim_{\epsilon\to0^+}\int_{U^\iota} (R_{g_\epsilon})_-d\mu_h = 0,\] where $(R_{g_\epsilon})_-$ is the negative part of the scalar curvature of $g_\epsilon^\iota$ and $U^\iota$ is a neighborhood of $S^\iota$ which only depends on $g$ and does not depend on $\epsilon$.
\end{lemma}

Secondly, we consider another particular case of Theorem \ref{PMT}, in which the metric is only singular at the corners.  we have the following lemma from \cite{Miao2002}:

\begin{lemma}\label{thm3.12}
Let $(M,g)$ be assumed as in Theorem \ref{PMT}, then there exists a family of metrics $g_\epsilon^0$ which is smooth on $(M\setminus S)\cup S^0$ such that\\
 (1) $g_\epsilon^0=g$ on $M\setminus S_{2\epsilon}^0$.\\
 (2) $g_\epsilon^0\to g$  on $M$ in $C^0$ norm as $\epsilon\to 0^+$ .\\
 (3) there exists an $A=A(g)>0$ which does not depend on $\epsilon$ and $x$, such that $\int_{B^h_1(x)}|\tilde\nabla g_\epsilon^0|^pd\mu_h<A$, where $p=\inf_{1\le \iota\le k_0} p_\iota$ and $B^h_1(x)$ is taken with respect to $h$.\\
 (4) Moreover, if $g$ has nonnegative scalar curvature, then \[\lowlim_{\epsilon\to0^+}\int_{U^0} (R_{g_\epsilon})_-d\mu_h = 0,\] where $(R_{g_\epsilon})_-$ is the negative part of the scalar curvature of $g_\epsilon^0$ and $U^0$ is a neighborhood of $S^0$ which only depends on $g$ and does not depend on $\epsilon$.
\end{lemma}
\begin{proof}
Though not has been explicitly stated in \cite{Miao2002}, this lemma is actually proved in it. Thus we only sketch how to get Lemma \ref{thm3.12} by results in \cite{Miao2002}.

Since the mollification and estimates  in \cite[Section 3]{Miao2002} are local, we apply Miao's construction, and get a family of metrics $g_\epsilon^0$ which is smooth on $(M\setminus S)\cup S^0$. Then \cite[Corollary 3.1]{Miao2002} says that $g_\epsilon^0$ satisfies (1) and (2) in Lemma \ref{thm3.12}, and the scalar curvature of $g_\epsilon^0$ is uniformly bounded on $S_{2\epsilon}^0$ with bounds depending only on $g$, but not on $\epsilon$. Thus we have
\begin{align}
\int_{U^0} (R_{g_\epsilon})_-d\mu_h 
&=\int_{S_{2\epsilon}} |R_{g_\epsilon}|d\mu_h \notag\\
&\le C \mu_h(S_{2\epsilon}).
\end{align}
Thus we have 
\begin{align}\label{39}
\lim_{\epsilon\to0^+}\int_{U^0} (R_{g_\epsilon})_-d\mu_h 
&=0.
\end{align}

By equations (4) and (6) in \cite{Miao2002},  using a Gaussian coordinate of $S^0$, we have
\begin{align}
g=\gamma(t)+dt^2
\end{align}
where , $\gamma(t)$ is a continuous path in the space of $C^2$ symmetric $(0,2)$ tensors on $S^0$.

Note the equation (16) in \cite{Miao2002}, $g_\epsilon^0$ is constructed as:

\begin{align}
g_\epsilon^0=\left\{
\begin{aligned}&\gamma_\epsilon(t)+dt^2,\quad (x,t)\in S^0\times(-\epsilon_0,\epsilon_0),\\
&g,\quad (x,t)\notin S^0\times(-\epsilon_0,\epsilon_0),
\end{aligned}\right.
\end{align}
where $(x,t)$ is the Gaussian coordinate, $\gamma_\epsilon(t)$ is a family of paths in the space of $C^0$ symmetric $(0,2)$ tensors on $S^0$.

By \cite[Lemma 3.2]{Miao2002}, $\gamma_\epsilon(t)$ converges to $\gamma(t)$ in $C^2$ norm. Thus the proposition (3) in Lemma \ref{thm3.12} holds.

By \eqref{39}, we have the proposition (4) in Lemma \ref{thm3.12} holds, thus the lemma is proved.
\end{proof}

Combining the constructions above, we can prove Lemma \ref{thm3.1}:
\begin{proof}[Proof of Lemma \ref{thm3.1}]
Note that the constructions in Lemma \ref{thm3.11} and Lemma \ref{thm3.12} only change the metric $g$ on the $2\epsilon$-neighborhood of the chosen singular set. And for $\epsilon$ small enough the intersections of those $2\epsilon$-neighborhoods are empty sets. Thus we define
\begin{align}
g_\epsilon=\left\{
\begin{aligned}
&g_\epsilon^\iota,\quad {\text on}\  S_{2\epsilon}^\iota,1\le\iota\le k_0,\\
&g_\epsilon^0,\quad {\text on}\  S_{2\epsilon}^0,\\
&g,\quad {\text on}\  M\setminus S_{2\epsilon},
\end{aligned}\right.
\end{align}
where $g_\epsilon^\iota$ and $g_\epsilon^0$ are the constructions in Lemma \ref{thm3.11} and Lemma \ref{thm3.12} respectively. Then $g_\epsilon$ satisfies the conclustion of Lemma \ref{thm3.1}.
\end{proof}

\subsection{Estimates on Ricci flow}
We use the Ricci flow approach in \cite{McSz} and \cite{ShiTam}. In this subsection we prove some estimates on the Ricci flow, and the main result is:

\begin{lemma}\label{thm6.2}
There exists an $\epsilon(n)>0$ such that, for any $n$-manifold $M$ with a complete and asymptotically flat metric $\hat g\in W_{\rm{loc}}^{1,p}(M)(p<n\le +\infty)$ and a metric $h\in C^\infty(M)$ which is $(1+\frac{\epsilon(n)}{2})$-fair to $\hat g$, satisfying that there exists an $A>0$ which is independent of $x$, such that $\int_{B^h_1(x)}|\tilde\nabla \hat g|^pd\mu_h<A$, where $B^h_1(x)$, $\tilde \nabla$ and $|\cdot|$ are taken with respect to $h$. Then there exists a $T=T(n,h,A,p)>0$, and a family of metrics $g(t)\in C^{\infty}(M\times(0,T])$, which is a $h$-flow of initial metric $\hat g$, $h$ is $(1+\epsilon(n))$-fair to $g(t)$, and
\begin{enumerate}
\item$\int_{B^h_1(x)}|\tilde\nabla g(t)|^pd\mu_h\leq10A$, $\forall t\in(0,T]$.\\
\item$|\tilde\nabla g|(t)\leq C(n,h,A,p)/t^{\frac{n}{2p}}$, $\forall t\in(0,T]$.\\
\item$|\tilde\nabla^2 g|(t)\leq C(n,h,A,p)/t^{\frac{n}{4p}+\frac{3}{4}}$, $\forall t\in(0,T]$.
\end{enumerate}
\end{lemma}
\begin{remark}
Such a $h$ always exists, moreover, we can even let $h=\hat g$ away from a compact set, by letting $h=\hat g_{\epsilon_0}$ for some $\epsilon_0$ small enough, where  $\hat g_{\epsilon_0}$ is the mollification metric in Lemma \ref{thm3.1}.
\end{remark}

\begin{remark}
The asymptotically flat condition could be weakend. However, assuming this condition is good enough for our use and the proof will be simpler.
\end{remark}

\begin{remark}
The existence of $h$-flow start from $C^0$ metrics has been proved in  \cite{simon2002deformation}, and we only have to prove estimates (1)-(3) in Lemma \ref{thm6.2}.
\end{remark}

Let us explain the conception of $h$-flow.  It was introduced in \cite{simon2002deformation}, in order to study the Ricci flow $\frac{\partial}{\partial t}g=-2Ric$ (we omit $t$ for simplicity if their is no ambiguity) introduced in \cite{Hamilton1982} starts from non-smooth metrics.

\begin{definition}
For a constant $\delta\geq 1$, a metric $h$ is called to be $\delta$-fair to $g$, if $h$ is $C^\infty$, 
\[\sup_M|\tilde\nabla^j \text{Rm}_h|=k_j<\infty,\]
and \[\delta^{-1}h\leq g\leq \delta h \qquad \text{on} \quad M,\]
where $\text{Rm}_h$ is the Riemannian curvature of $h$,  $\tilde\nabla$ and $|\cdot|$ are taken with respect to $h$.
\end{definition}
\begin{definition}\label{definition8}[$h$-flow]Fix a smooth background metric $h$ on $M$, the $h$-flow is a family of metrics $g$ satisfies \[\frac{\partial}{\partial t}g_{ij}=-2R_{ij}+ \nabla_iW_j+ \nabla_jW_i,\]
where $\nabla$ is taken with respect to $g$, \[W_i=g_{ik}g^{pq}(\Gamma_{pq}^k-\tilde\Gamma_{pq}^k),\]
$\Gamma$ and $\tilde\Gamma$ are the Christoffel symbols of $g$ and $h$. 
\end{definition}

\begin{remark}\label{rmk1}
Let  $\Psi_t$ be a family of diffeomorphisms given by
\begin{align*}
\frac{\partial}{\partial t}\Psi_t(x)=-W(\Psi_t(x),t), 
\end{align*}
then $\psi_t^*g(t)$ is a Ricci flow, see \cite{simon2002deformation}  or \cite{ShiTam}.
\end{remark}

The existence of $h$-flow was given by the following theorem in \cite{simon2002deformation}:
\begin{theorem}\label{theorem2.2} \cite[Theorem 1.1]{simon2002deformation}
There exists an $\epsilon(n)>0$ such that, for any $n$-manifold $M$ with a complete $C^0$ metric $\hat g$ and a $C^\infty$ metric $h$ which is $(1+\frac{\epsilon(n)}{2})$-fair to $\hat g$, there exists a $T=T(n,k_0)>0$ and a family of metrics $g(t)\in C^{\infty}(M\times(0,T]),t\in (0,T]$ which solves $h$-flow for $t\in(0,T]$, $h$ is $(1+\epsilon(n))$-fair to $g(t)$, and \[\lim_{t\to0^+}\sup_{x\in M}|g(x,t)-\hat g(x)|=0.\]
\[\sup_{M}|\tilde\nabla^ig|\leq\frac{c_i(n,k_1,...,k_i)}{t^{i/2}},\quad\forall t\in(0,T],\]
where $\tilde \nabla$ and $|\cdot|$ are taken respect to $h$.
\end{theorem}

To prove Lemma \ref{thm6.2}, we firstly prove (1) implies  (2):
\begin{lemma}\label{lemma3}
In the condition of Lemma \ref{thm6.2}, if for some $T\in(0,1]$, there holds $\int_{B^h_1(x)}|\tilde\nabla g(t)|^pd\mu_h\leq10A$, $\forall t\in(0,T]$, then for the same $T$, $|\tilde\nabla g|(t)\leq C(n,h,A,p)/t^{\frac{n}{2p}}$, $\forall t\in(0,T]$ also holds.
\end{lemma}
\begin{proof}
By \cite{Shi1989}, we have the evolution equation of $|\tilde\nabla g_{ij}|^2$
\begin{align*}
\frac{\partial}{\partial t}|\tilde\nabla g_{ij}|^2=&g^{\alpha\beta}\tilde\nabla_{\alpha}\tilde\nabla_{\beta}|\tilde\nabla g_{ij}|^2-2g^{\alpha\beta}\tilde\nabla_\alpha(\tilde\nabla g_{ij})\cdot\tilde\nabla_\beta(\tilde\nabla g_{ij})\\
&+\text{\~{Rm}}*g^{-1}*g^{-1}*g*\tilde\nabla g*\tilde\nabla g+g^{-1}*g*\tilde\nabla \text{\~{Rm}}*\tilde\nabla g\\
&+g^{-1}*g^{-1}*\tilde\nabla g*\tilde\nabla g*\tilde\nabla^2 g+g^{-1}*g^{-1}*g^{-1}*\tilde\nabla g*\tilde\nabla g*\tilde\nabla g*\tilde\nabla g,
\end{align*}
where the derivatives and norms are with respect to $h$, and \~{Rm} is the Riemannian curvature tensor of $h$.
Thus
\begin{align*}
\frac{\partial}{\partial t}|\tilde\nabla g_{ij}|^2-g^{\alpha\beta}\tilde\nabla_{\alpha}\tilde\nabla_{\beta}|\tilde\nabla g_{ij}|^2&\leq -C_1(n,h)|\tilde\nabla^2 g|^2\\
&\quad+C(n,h)|\tilde\nabla g|^2+C(n,h)|\tilde\nabla g|\\
&\quad+C(n)|\tilde\nabla g|^4+C_2(n)|\tilde\nabla g|^2|\tilde\nabla^2 g|\\
&\leq -C_1(n,h)|\tilde\nabla^2 g|^2\\
&\quad+C(n,h)|\tilde\nabla g|^2+C(n,h)|\tilde\nabla g|\\
&\quad+C(n)|\tilde\nabla g|^4+\frac{C_2(n)}{2\epsilon}|\tilde\nabla g|^4+\frac{C_2(n)\epsilon}{2}|\tilde\nabla^2 g|^2,\end{align*}
here and below $C$ and $C_i$s will denote positive constans depend at mostly on $n,h,A,p$ and $C$ can vary from line to line.
\begin{align*}
\frac{\partial}{\partial t}|\tilde\nabla g_{ij}|^2-g^{\alpha\beta}\tilde\nabla_{\alpha}\tilde\nabla_{\beta}|\tilde\nabla g_{ij}|^2
&\leq C(n,h)(\tilde\nabla g|+|\tilde\nabla g|^2+|\tilde\nabla g|^4),
\end{align*}
where we get the last inequality by taking $\epsilon=\frac{C_1}{C_2}$.\\
Denote $f=|\tilde\nabla g_{ij}|^2+1$, then we have
\[\frac{\partial}{\partial t}f-g^{\alpha\beta}\tilde\nabla_\alpha\tilde\nabla_\beta f\leq C(n,h)(f+f^{\frac{1}{2}}+f^2)\leq C(n,h)f(1+f)\]
Denote $v=1+f$, then 
\begin{equation}\label{evof}\frac{\partial}{\partial t}f-g^{\alpha\beta}\tilde\nabla_\alpha\tilde\nabla_\beta f\leq C(n,h)fv.
\end{equation}
And suppose that $T>0$ is a constant such that $\int_{B^h_1(x)}|\tilde\nabla g(t)|^pd\mu_h\leq10A$, $\forall t\in(0,T]$. Since $B^h_2(x)$ could be covered by finite many balls of radius $1$, and the amount of those balls is only depend on $n$ and $h$, we have
\[\int_{B^h_2(x)}v^{\frac{p}{2}}d\mu_h\leq C(n,h,A,p), \forall x\in M, \forall t\in[0,T],\]

For any $ x\in M$ and $0<r<R\leq2$, we let $\eta_0:\mathbb{R}\to [0,1]$ be a smooth function such that ${\text{supp}}\eta_0\in[0,R)$, $\eta_0\equiv1$ on $[0,r]$ and denote $\eta=\eta_0 \comp d(\cdot,x)$, where $d$ is the distance function with respect to metric $h$. We can assume $|\tilde\nabla \eta|\leq \frac{C(h)}{R-r}$, where the derivative is with respect to $h$.

For any $q\geq\frac{p}{2}-1$, we multiple $\eta^2 f^{q}$ to \eqref{evof} and integrate it, then we get
\begin{equation}
\int_{B^h_R(x)}\bigg(\frac{\partial}{\partial t}f-g^{\alpha\beta}\tilde\nabla_\alpha\tilde\nabla_\beta f\bigg)\eta^2f^{q}d\mu_h\leq C(n,h)\int_{B^h_R(x)}f^{q+1}v\eta^2d\mu_h
\end{equation}

Thus we have
\begin{align}\label{a0}
\frac{1}{q+1}\frac{\partial}{\partial t}\int_{B^h_R(x)}\eta^2f^{q+1}d\mu_h\leq
C(n,h)\bigg[&-\int_{B^h_R(x)}(\tilde\nabla_\alpha g^{\alpha\beta})(\tilde\nabla_\beta f)\eta^2f^{q}d\mu_h\notag\\
&-\int_{B^h_R(x)} g^{\alpha\beta}(\tilde\nabla_\beta f)\tilde\nabla_\alpha(\eta^2f^{q})d\mu_h\notag\\
&+\int_{B^h_R(x)}f^{q+1}v\eta^2d\mu_h\bigg].
\end{align}

For the first term in the right hand side, we have
\begin{align}\label{a1}
-\int_{B^h_R(x)}(\tilde\nabla_\alpha g^{\alpha\beta})(\tilde\nabla_\beta f)\eta^2f^{q}d\mu_h&\leq C(n,h)\int_{B^h_R(x)}f^{\frac{1}{2}}|\tilde\nabla f|\eta^2f^{q}d\mu_h\notag\\
&\leq C_1\epsilon\int_{B^h_R(x)}|\tilde\nabla f|^2\eta^2f^{q-1}d\mu_h\notag\\
&\quad +\frac{C_1}{\epsilon}\int_{B^h_R(x)}\eta^2f^{q+2}d\mu_h,
\end{align}
for any $\epsilon>0$.

And for the second term in the right hand side, we have
\begin{align}\label{a2}
&\quad -\int_{B^h_R(x)} g^{\alpha\beta}(\tilde\nabla_\beta f)\tilde\nabla_\alpha(\eta^2f^{q})d\mu_h\notag\\
&=-q\int_{B^h_R(x)} (g^{\alpha\beta}\tilde\nabla_\beta f\tilde\nabla_\alpha f)f^{q-1}\eta^2d\mu_h-2\int_{B^h_R(x)} g^{\alpha\beta}(\tilde\nabla_\beta f)f^{q}(\tilde\nabla_\alpha \eta)\eta d\mu_h\notag\\
&\leq -C_2(n,h)q\int_{B^h_R(x)}|\tilde\nabla f|^2f^{q-1}\eta^2d\mu_h+C_3(n,h)\int_{B^h_R(x)}|\tilde\nabla f|f^{q}|\tilde\nabla \eta|\eta d\mu_h\notag\\
&\leq -C_2q\int_{B^h_R(x)}|\tilde\nabla f|^2f^{q-1}\eta^2d\mu_h+C_3\frac{C_2q}{2C_3}\int_{B^h_R(x)}|\tilde\nabla f|^2f^{q-1}\eta^2d\mu_h\notag\\
&\quad+C_3\frac{C_3}{2C_2q}\int_{B^h_R(x)}f^{q+1}|\tilde\nabla\eta|^2d\mu_h\notag\\
&\leq -C_4q\int_{B^h_R(x)}|\tilde\nabla f|^2f^{q-1}\eta^2d\mu_h+\frac{C_4}{q}\int_{B^h_R(x)}f^{q+1}|\tilde\nabla\eta|^2d\mu_h.
\end{align}
Take $\epsilon=\frac{C_4q}{2C_1}$ in \eqref{a1}, recall that $v=f+1$ and hence $f^{q+2}\leq f^{q+1}v$, combining \eqref{a0}, \eqref{a1} and \eqref{a2} we have
\begin{align}
&\quad\frac{1}{q+1}\frac{\partial}{\partial t}\int_{B^h_R(x)}\eta^2f^{q+1}d\mu_h+C(n,h)q\int_{B^h_R(x)}|\tilde\nabla f|^2f^{q-1}\eta^2d\mu_h\notag\\
&\leq \frac{C(n,h)}{q}\int_{B^h_R(x)}f^{q+1}|\tilde\nabla\eta|^2d\mu_h+C(n,h)\int_{B^h_R(x)}f^{q+1}v\eta^2d\mu_h
\end{align}
Since
\begin{align*}
&\quad\int_{B^h_R(x)}\big|\tilde\nabla(f^{\frac{q+1}{2}}\eta)\big|^2d\mu_h\\
&=\int_{B^h_R(x)}\big|\tilde\nabla\eta f^{\frac{q+1}{2}}+\frac{q+1}{2}\tilde\nabla f\cdot f^{\frac{q-1}{2}}\eta\big|^2d\mu_h\\
&\leq\frac{(q+1)^2}{2}\int_{B^h_R(x)}|\tilde\nabla f|^2f^{q-1}\eta^2d\mu_h+2\int_{B^h_R(x)}|\tilde\nabla \eta|^2f^{q+1}d\mu_h,
\end{align*}
we have
\begin{align}
&\quad\frac{1}{q+1}\frac{\partial}{\partial t}\int_{B^h_R(x)}\eta^2f^{q+1}d\mu_h+C(n,h)\frac{q}{(q+1)^2}\int_{B^h_R(x)}\big|\tilde\nabla(f^{\frac{q+1}{2}}\eta)\big|^2d\mu_h\notag\\
&\leq \frac{C(n,h)}{q}\int_{B^h_R(x)}f^{q+1}|\tilde\nabla\eta|^2d\mu_h+C(n,h)(1+\frac{q}{(q+1)^2})\int_{B^h_R(x)}f^{q+1}v\eta^2d\mu_h.
\end{align}
Since $q\geq\frac{p}{2}-1$, $\frac{q}{q+1}\leq1$ and $|\tilde\nabla \eta|<\frac{C(n,h)}{R-r}$,
we have
\begin{align}\label{aa0}
&\quad\frac{\partial}{\partial t}\int_{B^h_R(x)}\eta^2f^{q+1}d\mu_h+\int_{B^h_R(x)}\big|\tilde\nabla(f^{\frac{q+1}{2}}\eta)\big|^2d\mu_h\notag\\
&\leq \frac{C(n,h,p)q}{(R-r)^2}\int_{B^h_R(x)}f^{q+1}d\mu_h+C(n,h,p)q\int_{B^h_R(x)}f^{q+1}v\eta^2d\mu_h.
\end{align}
For the last term, use H\"older inequality and since $p>n$ we can use interpolation inequality, then we have
\begin{align}\label{aa1}
\int_{B^h_R(x)}f^{q+1}v\eta^2d\mu_h&\leq \bigg(\int_{B^h_R(x)}v^{\frac{p}{2}}d\mu_h\bigg)^{\frac{2}{p}}\bigg(\int_{B^h_R(x)}(f^{q+1}\eta^2)^{\frac{p}{p-2}}d\mu_h\bigg)^{\frac{p-2}{p}}\notag\\
&\leq C(n,h,A,p)\bigg(\epsilon\int_{B^h_R(x)}(f^{q+1}\eta^2)^{\frac{n}{n-2}}d\mu_h\bigg)^{\frac{n-2}{n}}+\epsilon^{-\mu}\int_{B^h_R(x)}f^{q+1}\eta^2d\mu_h\bigg],
\end{align}
where $\mu=(\frac{n-2}{n}-\frac{p-2}{p})/(\frac{p-2}{p}-1)=\frac{p-n}{n}$.
We also have the Sobolev inequality
\begin{equation}\label{aa2}
\int_{B^h_R(x)}\big|\tilde\nabla(f^{\frac{q+1}{2}}\eta)\big|^2d\mu_h\geq C(n,h)\bigg(\int_{B^h_R(x)}(f^{q+1}\eta^2)^{\frac{n}{n-2}}d\mu_h\bigg)^{\frac{n-2}{n}}.
\end{equation}
Take $\epsilon$ in \eqref{aa1} to be some appropriate constant whose form is like $\frac{C}{q}$, where the $C$ depending on constant $C$s in inequations above, combine \eqref{aa0}, \eqref{aa1} and \eqref{aa2}, we have
\begin{align}\label{15}
\frac{\partial}{\partial t}\int_{B^h_r(x)}f^{q+1}d\mu_h+\bigg(\int_{B^h_r(x)}(f^{q+1})^{\frac{n}{n-2}}d\mu_h\bigg)^{\frac{n-2}{n}}\notag\\
\leq\frac{C(n,h,A,p)(q+q^{1+\mu})}{(R-r)^2}\int_{B^h_R(x)}f^{q+1}d\mu_h.
\end{align}
For any $0<t'<t''<T'\leq T\leq1$, let
\[\psi(t)=\left\{
\begin{array}{rl}
0, & \text{if } 0\leq t\leq t'\\
\frac{t-t'}{t''-t'}, & \text{if } t'\leq t\leq t'',\\
1, & \text{if } t''\leq t\leq T.
\end{array} \right. \]
Multiplying (\ref{15}) by $\psi$, we get
\begin{align}\label{16}
\frac{\partial}{\partial t}\int_{B^h_r(x)}\psi f^{q+1}d\mu_h+\psi\bigg(\int_{B^h_r(x)}(f^{q+1})^{\frac{n}{n-2}}d\mu_h\bigg)^{\frac{n-2}{n}}\notag\\
\leq[\frac{C(n,h,A,p)(q+q^{1+\mu})}{(R-r)^2}\psi+\psi']\int_{B^h_R(x)}f^{q+1}d\mu_h.
\end{align}
Integrating this with respect to $t$, we get
\begin{align}
\sup_{t\in[t'',T']} \int_{B^h_r(x)}f^{q+1}d\mu_h+\int_{t''}^{T'}\bigg(\int_{B^h_r(x)}(f^{q+1})^{\frac{n}{n-2}}d\mu_h\bigg)^{\frac{n-2}{n}}dt\notag\\
\leq[\frac{C(n,h,A,p)(q+q^{1+\mu})}{(R-r)^2}+\frac{1}{t''-t'}]\int_{t'}^{T'}\int_{B^h_R(x)}f^{q+1}d\mu_hdt.
\end{align}
\begin{align}\label{14}
&\quad\int_{t''}^{T'}\int_{B^h_r(x)}f^{(q+1)(1+\frac{2}{n})}d\mu_hdt\notag\\
&\leq \int_{t''}^{T'}\bigg(\int_{B^h_r(x)}f^{q+1}d\mu_h\bigg)^{\frac{2}{n}}\bigg(\int_{B^h_r(x)}f^{(q+1)\frac{n}{n-2}}d\mu_h\bigg)^{\frac{n-2}{n}}dt\notag\\
&\leq\sup_{t\in[t'',{T'}]}\bigg(\int_{B^h_r(x)}f^{q+1}d\mu_h\bigg)^{\frac{2}{n}}\int_{t''}^{T'}\bigg(\int_{B^h_r(x)}f^{(q+1)\frac{n}{n-2}}d\mu_h\bigg)^{\frac{n-2}{n}}dt\notag\\
&\leq[\frac{C(n,h,A,p)(q+q^{1+\mu})}{(R-r)^2}+\frac{1}{t''-t'}]^{1+\frac{n}{2}}\bigg(\int_{t'}^{T'}\int_{B^h_R(x)}f^{q+1}d\mu_hdt\bigg)^{1+\frac{2}{n}}
\end{align}
Denote \[H(q,\rho,\tau)=\bigg(\int_{\tau}^{T'}\int_{B^h_\rho(x)}f^{q}d\mu_hdt\bigg)^{\frac{1}{q}},\forall q\geq\frac{p}{2}, 0<r<1, 0<\tau<T',\]
Then equation (\ref{14}) can be shortly writen as
\[H(q(1+\frac{2}{n}),r,t'')\leq[\frac{C(n,h,A,p)}{(R-r)^2}(q+q^{1+\mu})+\frac{1}{t''-t'}]^\frac{1}{q} H(q,R,t')\]
Fix $0<t_0<t_1<T'\leq1$, $q_0\geq2$ and set $\chi=1+\frac{2}{n}$, $q_k=q_0\chi^k$, $\tau_k=t_0+(1-\frac{1}{\chi^k})(t_1-t_0)$, $r_k=\frac{1}{2}+\frac{1}{2^{k+1}}$. Then we have
\begin{align*}H(q_{k+1},r_{k+1},\tau_{k+1})&\leq[\frac{C(n,h,A,p)}{(r_k-r_{k+1})^2}(q_k+q_k^{1+\mu})+\frac{1}{t_1-t_0}\frac{\chi}{\chi-1}\chi^k]^\frac{1}{q_k} H(q_k,r_k,\tau_k)\\
&= [C(n,h,A,p)4^{k+1}q_0\chi^k+q_0^{1+\mu}\chi^{k(1+\mu)}+\frac{1}{t_1-t_0}\frac{n+2}{2}\chi^k]^\frac{1}{q_k}\chi^{\frac{k}{q_k}} H(q_k,r_k,\tau_k)\\
&\leq [\frac{C(n,h,A,p,q_0)}{t_1-t_0}]^\frac{1}{q_k}\chi^{\frac{k(1+\mu)}{q_k}} H(q_k,r_k,\tau_k),
\end{align*}
where in the last inequality we use $0<t_0<t_1<T$
By iteration, we get
\begin{align*}
H(q_{m+1},r_{m+1},\tau_{m+1})
&\leq [\frac{C(n,h,A,p,q_0)}{t_1-t_0}]^{\sum_{k=0}^m\frac{1}{q_k}}\chi^{\sum_{k=0}^m\frac{k(1+\mu)}{q_k}} H(q_0,r_0,\tau_0)\\
&\leq C(n,h,A,p,q_0)(\frac{1}{t_1-t_0})^{\frac{n+2}{2q_0}}H(q_0,r_0,\tau_0).
\end{align*}
Letting $m\to\infty$, we get
\[H(p_\infty,r_\infty,\tau_\infty)\leq C(n,h,A,p,q_0)(\frac{1}{t_1-t_0})^{\frac{n+2}{2q_0}}H(q_0,r_0,\tau_0), \forall q_0\geq\frac{p}{2}.\]
Letting $q_0=\frac{p}{2}$, since $r_\infty=\frac{1}{2}$, $\tau_\infty=t_1$, we have
\[H(\infty,\frac{1}{2},t_1)\leq C(n,h,A,p)(\frac{1}{t_1})^{\frac{n+2}{p}}H(\frac{p}{2},1,0),\]
i.e.
\[\sup_{(y,t)\in B^h_{\frac{1}{2}}(x)\times[t_1,T']}f(y,t)\leq C(n,h,A,p)\frac{1}{(t_1-t_0)^{\frac{n+2}{p}}}\bigg(\int_{t_0}^{T'}A^{\frac{p}{2}}dt\bigg)^{\frac{2}{p}}.\]
Letting $t_1\to T'$ and $t_0=T'/2$, we get
\[\sup_{y\in B^h_{\frac{1}{2}}(x)}f(y,T')\leq C(n,h,A,p)\frac{1}{(T')^{\frac{n}{p}}},\forall T'\in(0,T], x\in M.\]
Thus we have
\[|\tilde\nabla g|(t)\leq C(n,h,A,p)/t^{\frac{n}{2p}}, \forall t\in(0,T],\]
thus the lemma is proved.
\end{proof}

Now we can prove that (1)-(3) of Lemma \ref{thm6.2} holds:
\begin{proof}[proof of Lemma \ref{thm6.2}]
The first half of this lemma was given in \cite[Theorem 1.1]{simon2002deformation}, and we only need to prove conclusions (1), (2) and (3). In fact, we can assume that $\hat g$ is smooth without loss of generality. This is because for a $W_{\text{loc}}^{1,p}$ metric $\hat g$, if $\int_{B^h_1(x)}|\tilde\nabla \hat g|^pd\mu_h<A$, then   we can approximate it  by a family of smooth metrics $\hat g_\epsilon$ such that $\int_{B^h_1(x)}|\tilde\nabla \hat g_\epsilon|^pd\mu_h<C(n,h)A$, (see the proof of Lemma \ref{thm3.1}). And then Theorem \ref{theorem2.2} gives the $h$-flow $g_\epsilon(t)$, which is uniformly $C_{\text{loc}}^\infty$ bounded for any $[a,b]\subset (0,T]$. Then by a diagonal subsequent argument and by Arzel\`a-Ascoli theorem, we get the $h$-flow $g(t)$ of initial metric $\hat g$. If we have estimates (1)-(3) in Lemma \ref{thm6.2} for $g_\epsilon(t)$, then (1)-(3) also holds for $g(t)$ by taking limit. Therefore we can assume that $\hat g$ is smooth without loss of generality. 

For any $ x\in M$, we let $\eta_0:\mathbb{R}\to [0,1]$ be a smooth function such that ${\text{supp}}\eta_0\in[0,2)$, $\eta_0\equiv1$ on $[0,1]$ and denote $\eta=\eta_0 \comp d_h(\cdot,x)$, where $d_h$ is the distance function taken with respect to metric $h$. We can assume $|\tilde\nabla \eta|\leq C(h)$, where the derivative is taken with respect to $h$. Denote $f=|\tilde\nabla g_{ij}|^2+1$, and let 
\[\phi(t)=\sup_{x\in M}\int_{B^h_2(x)}f^{\frac{p}{2}}\eta^2d\mu_h.\]

By \cite[Lemma 7.1]{ShiTam} (or \cite[Theorem 5]{McSz}), we know that $g(t)$ is asymptotically flat and thus $f$ tends to $1$ when $x\to\infty$, thus the supremum is attained in some point. We assume that the supremum is attained at $x_0$ for time $t$. And for some other time $t'$, we denote the  correspongding point where the supremum is attained by $x'_0$, then we have $\int_{B^h_{2}(x'_0)}f^{\frac{p}{2}}(t')\eta^2(t') d\mu_h\geq \int_{B^h_{2}(x_0)}f^{\frac{p}{2}}(t')\eta^2(t') d\mu_h$. 
\begin{align}\label{b0}
\phi'_-(t)&=\lim_{t'\to t^-}\frac{\phi(t)-\phi(t')}{t-t'}\notag\\
&=\lim_{t'\to t^-}\frac{\int_{B^h_2(x_0)}f^\frac{p}{2}(t)\eta^2(t) d\mu_h-\int_{B^h_2(x'_0)}f^\frac{p}{2}(t')\eta^2(t') d\mu_h}{t-t'}\notag\\
&\leq\lim_{t'\to t^-}\frac{\int_{B^h_2(x_0)}f^\frac{p}{2}(t)\eta^2(t) d\mu_h-\int_{B^h_2(x_0)}f^\frac{p}{2}(t')\eta^2(t') d\mu_h}{t-t'}\notag\\
&=\frac{\partial}{\partial t}\int_{B^h_2(x_0)}f^\frac{p}{2}\eta^2 d\mu_h,
\end{align}
where $\phi'_-(t)$ is the left derivative of $\phi(t)$.

Take $q=p/2-1$, $r=1$ and $R=2$ in \eqref{aa0}, we have
\begin{align}\label{b1}
\phi'_-(t)\leq C(n,h,p)
\int_{B^h_2(x)}f^{p/2}d\mu_h+C(n,h,p)\int_{B^h_2(x)}f^{p/2}v\eta^2d\mu_h,
\end{align}
where $v=1+f$.

Let
\[\mathcal{T}=\bigg\{T\in(0,1]\big|\int_{B^h_1(x)}|\tilde\nabla g(t)|^pd\hat g\leq10A,\forall x\in M, \forall t\in[0,T]\bigg\}.\]

Since we have assumed that $\hat g$ is smooth, thus $g(t)$ converges to $\hat g$ in $C^\infty$ norm on any compact subset of $M$ as $t$ tends to $0^+$, thus $\mathcal{T}\neq\emptyset$. We denote $T_{\max}=\sup\mathcal{T}$.
By Lemma \ref{lemma3}, we have $v\le C(n,h,A,p)/t^{n/p}$ for any $t\in(0,T_{\max}]$, thus we have
\begin{equation}
\phi'_-(t)\leq C(n,h,p)\int_{B^h_2(x)}f^{p/2}d\mu_h+\frac{C(n,h,A,p)}{t^{n/p}}\int_{B^h_2(x)}f^{p/2}\eta^2d\mu_h, \forall t\in(0,T_{\max}].
\end{equation}

We can choose $N$ balls with radius $1$ to cover $B^h_1(x_0)$, where $N$ only depends on $n$ and $h$. The integral of $f^{p/2}$ on each of these balls is less than $\phi(t)$, thus we get
\begin{equation*}
\phi'_-(t)\leq C(n,h,p)\phi(t)+\frac{C(n,h,A,p)}{t^{n/p}}\phi(t),\quad \forall t\in(0,T_{\max}].
\end{equation*}

Since $T_{\max}\leq 1$,
we have
\[\phi'_-(t)\leq \frac{C(n,h,A,p)}{t^{n/p}}\phi(t),\quad \forall t\in(0,T_{\max}],\]

Thus we have
\[(\log\phi(t))'_-\leq\frac {C(n,h,A,p)}{t^{n/p}},\quad \forall t\in(0,T_{\max}].\]

Since $\phi$ is Lipschitz, we get
\[\log\phi(t)\leq \log\phi(0)+C(n,h,A,p)t^{1-n/p},\quad \forall t\in(0,T_{\max}],\]
and thus
\[\phi(t)\leq\phi(0)e^{C(n,h,A,p)t^{1-n/p}},\quad \forall t\in(0,T_{\max}].\]

Note $\phi(0)\leq C(n,h)A$, thus 
\begin{equation}\label{21}\int_{B^h_1(x)}|\tilde\nabla g(t)|^pd\mu_h\leq \phi(t)\leq C(n,h)e^{C(n,h,A,p)t^{1-n/p}}A,\quad \forall t\in(0,T_{\max}].
\end{equation}

Suppose that if we fix $n$, $h$, $p$ and $A$, there exists a sequence of initial metric $\{\hat g_m\}_{m=1}^\infty$, such that each of the metric $\hat g_m$ satisfies the condition of the theorem and the corresponding $T_{\max,m}$ tends to $0$. But by equation (\ref{21}) we know that if $T_{\max}$ satisfies $C(n,h)e^{C(n,h,A,p)T_{\max}^{1-n/p}}\leq 5$, then the maximal time interval $(0,T_{\max}]$ could be extended to $(0,T_{\max}+\delta]$, for some $\delta>0$ small enough, which is a contradiction. Hence $T_{\max}$ is only depend on $n$, $h$, $p$ and $A$. Thus we get the conclusion (1) and (2).

To prove (3), recall that Simon's result, Theorem \ref{theorem2.2}, gives
\[\sup_{M}|\tilde\nabla^ig|\leq\frac{c_i(n,h)}{t^{i/2}},\quad\forall t\in(0,T].\]

Recall Gagliardo-Nirenberg interpolation inequality \cite[section 6.8]{gilbarg2015elliptic}
\[\sup_\Omega |\partial u|\leq2|\sup_\Omega u|^{\frac{1}{2}}|\sup_\Omega \partial^2u|^{\frac{1}{2}},\forall \text{open subset }{}\Omega\in\mathbb{R}^n, u\in C^2(\Omega). \]

Since we can choose a finite atlas on $(M,h)$ such that $h$ is uniformly equivalent to the Euclidean metric of each chart, the Gagliardo-Nirenberg interpolation inequality can be passed to $(M,h)$ with the coefficient $2$ replaced by some $C(h).$
Hence using (2) we get
\[|\tilde\nabla^2 g|(t)\leq C(n,h,A,p)\sup_M|\tilde\nabla g(t)|^{\frac{1}{2}}\sup_M|\tilde\nabla^3g(t)|^{\frac{1}{2}}
\leq \frac{C(n,h,A,p)}{t^{\frac{n}{4p}+\frac{3}{4}}},\]
thus the lemma is proved.
\end{proof}

\subsection{Some geometric quantity along the flow}
In this subsection we study the ADM mass and the scalar curvature along the flow.

For the asymptotically flatness and the ADM mass, we have the following result from \cite{ShiTam}:
\begin{lemma}\label{ST1}
\cite[Lemma 7.1]{ShiTam}
Let $(M,\hat g)$ be assumed as in Theorem \ref{thm6.2},  $\hat g_\epsilon$ be the mollification of $\hat g$ given in Lemma \ref{thm3.1}. Let $g_\epsilon(t)$ and $g(t)$ be the $h$-flow of initial metric $\hat g_\epsilon$ and $\hat g$. Then 
\begin{enumerate}
\item $g_\epsilon(t)$ and $g(t)$ are asymptotically flat
\item $m(\hat g)=m(\hat g_\epsilon)=m(g_\epsilon(t)$, where $m(\hat g), m(\hat g_\epsilon)$ and $m(g_\epsilon(t))$ are the ADM mass of $\hat g$, $\hat g_\epsilon$ and $g_\epsilon(t)$.
	\end{enumerate}
\end{lemma}
\begin{remark}
In \cite{ShiTam}, the dimensional condition on the singular set is only used in theire Lemma 4.5, in order to prove the nonnegativity of scalar curvature of $g(t)$. For their Lemma 7.1, the condition that  $\hat g\in W_{\rm{loc}}^{1,p}(M)(p<n\le +\infty)\cap C^\infty(M\setminus S)$ for any compact $S\subset M$ and $\hat g$ is asymptotically flat is sufficient.
\end{remark}

For the scalar curvature, we have
\begin{lemma}\label{NNSC}
Let $(M,\hat g)$ be assumed as the same as $(M,g)$ in Theorem \ref{PMT}, and let $g(t)$ be the $h$-flow with initial metric $\hat g$, then $g(t)$ has nonnegative scalar curvature.
\end{lemma}

We use techniques developed in Subsection 2.1 and Subsection 2.2 to prove Lemma \ref{NNSC}, firstly we prove the following lemma:
\begin{lemma}\label{m}
Let $M, \hat g, \hat g_\epsilon, T, g_\epsilon(t), g(t)$ be as in Lemma \ref{ST1}, then 
\[\lowlim_{\epsilon\to0^+}\int_{B^h_{1/2}(x)} \left|(R_{g_\epsilon(t)})_-\right|d\mu_{g_\epsilon(t)}=0,\forall t\in(0,T],\forall x\in M,\]
where $R_{g_\epsilon(t)}$ is the scalar curvature of $g_\epsilon(t)$, $(R_{g_\epsilon(t)})_-$ is its negative part, $d\mu_{g_\epsilon(t)}$ is the volume form of $g_\epsilon(t)$ and the ball is taken with respect to $h$.
\end{lemma}
\begin{remark}
Note that by taking limit, 
Lemma \ref{m} tells us that $R_{g(t)}$ is nonnegative.
\end{remark}
\begin{proof}
For any $\delta>0$, denote 
\[v=\sqrt{R_{g_\epsilon(t)}^2+\delta^2}-R_{g_\epsilon(t)}.\]

In this prove we will use the evolution equation along Ricci flow.  Let $\Psi_{\epsilon,t}$ be the family of diffeomorphisms in Remark \ref{rmk1} such that $\Psi_{\epsilon,0}$ are the identity map and $\Psi_{\epsilon,t}^*g_\epsilon(t)$ are Ricci flow. 
then
$\Psi_{\epsilon,t}^*v=v\circ\Psi_{\epsilon,t}=\sqrt{R_{\Psi_{\epsilon,t}^*g_\epsilon(t)}^2+\delta^2}-R_{\Psi_{\epsilon,t}^*g_\epsilon(t)}$ and 
 $\Psi_{\epsilon,t}^*v\to 2(R_{\Psi_{\epsilon,t}^*g_\epsilon(t)})_-$ as $\delta\to 0^+$. And as is shown in the proof of \cite[Lemma 13]{McSz}, (in \cite[Lemma 13]{McSz}, their $g(t)$ is the Ricci flow ,which could be seen as our $\Psi_{\epsilon,t}^*g_\epsilon(t)$ in fact), $\Psi_{\epsilon,t}^*v$ satisfies
\[\frac{\partial \Psi_{\epsilon,t}^*v}{\partial t}\leq \Delta_{\Psi_{\epsilon,t}^*g_\epsilon(t)} \Psi_{\epsilon,t}^*v.\]

For any $ x\in M$, we let $\varphi_0:\mathbb{R}\to [0,1]$ be a smooth function such that ${\text{supp}}\varphi_0\in[0,1)$, $\varphi_0\equiv1$ on $[0,\frac{1}{2}]$ and denote $\varphi_x=\varphi_0 \comp d_h(\cdot,x)$, where $d_h$ is the distance function taken with respect to metric $h$. Then ${\text{supp}}\varphi_x\subset B^h_1(x)$, where $B^h$ is the geodesic ball taken with respect to $h$. $|\tilde\nabla \varphi_x|\le C(n,h) $, and by Lemma \ref{thm6.2}, we have $\Delta_{g_\epsilon(t)}\varphi_x\leq C(n,h,A,p)/t^{n/2p}$, where $C(n,h,A,p)$ is independent of $x$, $\epsilon$ and $t$ and can vary from line to line.

Note that by Lemma \ref{theorem2.2}, $h$ is $(1+\epsilon(n))$-fair to $g_\epsilon(t)(t\in[0,T])$, thus 
\begin{align}\label{B1}
B^h_1(x)\subset B^{g_\epsilon(t)}_{C_1(n)}(x),
\end{align}
 where $B^h$ and $B^{g_\epsilon(t)}$ are the geodesic ball of $h$ and $g_\epsilon(t)$ respectively, $C_1$ is a positive constant which is independent of $\epsilon$ and $t$. 

 And the definition of pullback gives 
 \begin{align}\label{B2}
 \Psi_{\epsilon,t}^{-1}(B^{g_\epsilon(t)}_{C_1(n)}(x))=B^{\Psi_{\epsilon,t}^*g_\epsilon(t)}_{C_1(n)}(\Psi_{\epsilon,t}^{-1}(x)).
 \end{align} 

To estimates the ball after pullback, note that Lemma \ref{thm6.2} gives 

\begin{align}
|\text{Rm}_{g_\epsilon(t)}|\le C(n,h,A,p)/t^{\frac{n}{4p}+\frac{3}{4}},\forall t\in(0,T].
\end{align}

And since $h$ is $(1+\epsilon(n))$-fair to $g_\epsilon(t)(t\in[0,T])$, we have 
\begin{align}
|\text{Rm}_{g_\epsilon(t)}|_{g_\epsilon(t)}\le C(n,h,A,p)/t^{\frac{n}{4p}+\frac{3}{4}},\forall t\in(0,T].
\end{align}
 where $|\cdot|_{g_\epsilon(t)}$ is the norm taken with respect to $g_\epsilon(t)$. 

 And applying the pullback we have 
 \begin{align}\label{1ot}
 |\text{Rm}_{\Psi_{\epsilon,t}^*g_\epsilon(t)}|_{\Psi_{\epsilon,t}^*g_\epsilon(t)}\le C(n,h,A,p)/t^{\frac{n}{4p}+\frac{3}{4}},\forall t\in(0,T].
 \end{align}

  Since $1/t^{\frac{n}{4p}+\frac{3}{4}}$ is integrable on $[0,T]$, using the Ricci flow equation, \eqref{1ot} implies that 
  \begin{align}
C^{-1}(n,h,A,p)\hat g_\epsilon\le \Psi_{\epsilon,t}^*g_\epsilon(t)\le C(n,h,A,p)\hat g_\epsilon.
\end{align}

Since $C^{-1}(n)h\le\hat g_\epsilon\le C(n)h$ for $\epsilon$ small enough, we have
 \begin{align}
C^{-1}(n,h,A,p)h\le \Psi_{\epsilon,t}^*g_\epsilon(t)\le C(n,h,A,p)h,
\end{align}
which implies 
\begin{align}\label{B3}
 B^{\Psi_{\epsilon,t}^*g_\epsilon(t)}_{C_1(n)}(\Psi_{\epsilon,t}^{-1}(x))\subset B^h_{C_2(n,h,A,p)}(\Psi_{\epsilon,t}^{-1}(x)),
 \end{align} 
where $C_2$ is a positve constant which is independent of $\epsilon$ and $t$.

Combine \eqref{B1}, \eqref{B2} and \eqref{B3}, we have ${\text{supp}}\Psi_{\epsilon,t}^*\varphi_x\subset \Psi_{\epsilon,t}^{-1}(B^h_1(x) \subset  B^h_{C_2(n,h,A,p)}(\Psi_{\epsilon,t}^{-1}(x))$. Similarly, we also have $\Psi_{\epsilon,t}^*\varphi_x=1$ on $B^h_{C_3(n,h,A,p)}(\Psi_{\epsilon,t}^{-1}(x))$ for some positive constant $C_3$ which is independent of $\epsilon$ and $t.$

Let
\begin{align}\label{Dphi}
\phi(t)=\sup_{x\in M}\int_Mv\varphi_x d\mu_{g_\epsilon(t)}=\sup_{x\in M}\int_M\Psi_{\epsilon,t}^*v\Psi_{\epsilon,t}^*\varphi_x d\mu_{\Psi_{\epsilon,t}^*g_\epsilon(t)}, t \in [0,T].
 \end{align}

Since by Lemma \ref{ST1} we have $ g_\epsilon(t)$ is asymptotically flat, there exists a compact $K_{\epsilon,t}$, such that the integral $\int_M v \varphi_x d\mu_{ g_\epsilon(t)}=\int_M\Psi_{\epsilon,t}^*v\Psi_{\epsilon,t}^*\varphi_x d\mu_{\Psi_{\epsilon,t}^*g_\epsilon(t)}$ is bounded by $\delta$ for $x$ belongs to $M\backslash K_{\epsilon,t}$, so the supremum is well defined.

Claim: there exists a constant $C=C(n,h,A,p)>0$ independ of $\delta$, $\epsilon$ and $t$, such that
\[\phi(t)\leq C(n,h,A,p)(\phi(0)+2\delta)\]

In fact, if the supremum could not be attained in $K_{\epsilon,t}$ for some $t$, then $\phi(t)\leq \delta<2\delta$ and the claim holds.

Otherwise, suppose that the supremum is attained at some $x_0\in K_{\epsilon,t}$ for time $t$, and suppose that the maximum value is greater than $2\delta$, then for any other time $t'$ close enough to $t$, the supremum is greater than $\delta$, hence there is also a correspongding point where the supremum is attained, which we denote by $x'_0$, then we have $\int_M\Psi_{\epsilon,t}^*v(t')\Psi_{\epsilon,t}^*\varphi_{x'_0}(t') d\mu_{\Psi_{\epsilon,t}^*g_\epsilon(t')}\geq \int_M\Psi_{\epsilon,t}^*v(t')\Psi_{\epsilon,t}^*\varphi_{x_0}(t') d\mu_{\Psi_{\epsilon,t}^*g_\epsilon(t')}$. Hence
\begin{align*}
\phi'_-(t)&=\lim_{t'\to t^-}\frac{\phi(t)-\phi(t')}{t-t'}\\
&=\lim_{t'\to t^-}\frac{\int_M\Psi_{\epsilon,t}^*v(t)\Psi_{\epsilon,t}^*\varphi_{x'_0}(t) d\mu_{\Psi_{\epsilon,t}^*g_\epsilon(t)}-\int_M\Psi_{\epsilon,t}^*v(t')\Psi_{\epsilon,t}^*\varphi_{x'_0}(t') d\mu_{g_\epsilon(t')}}{t-t'}\\
&\leq\lim_{t'\to t^-}\frac{\int_M\Psi_{\epsilon,t}^*v(t)\Psi_{\epsilon,t}^*\varphi_{x'_0}(t) d\mu_{\Psi_{\epsilon,t}^*g_\epsilon(t)}-\int_M\Psi_{\epsilon,t}^*v(t')\Psi_{\epsilon,t}^*\varphi_{x'_0}(t') d\mu_{g_\epsilon(t')}}{t-t'}\\
&=\int_M\frac{\partial}{\partial t}\left(\Psi_{\epsilon,t}^*v\Psi_{\epsilon,t}^*\varphi_{x'_0} d\mu_{\Psi_{\epsilon,t}^*g_\epsilon(t)}\right),
\end{align*}
where $\phi'_-(t)$ is the left derivative of $\phi(t)$.
Using the standard evolution equation for Ricci flow, $\frac{\partial }{\partial t}d\mu_{\Psi_{\epsilon,t}^*g_\epsilon(t)}=-R_{\Psi_{\epsilon,t}^*g_\epsilon(t)}d\mu_{\Psi_{\epsilon,t}^*g_\epsilon(t)}$, where $R_{\Psi_{\epsilon,t}^*g_\epsilon(t)}$ is the scalar curvature of $\Psi_{\epsilon,t}^*g_\epsilon(t)$, we have
\begin{align}\label{R1}
\phi'_-(t)&\leq\int_M\frac{\partial}{\partial t}\Psi_{\epsilon,t}^*v\Psi_{\epsilon,t}^*\varphi_{x_0}-\Psi_{\epsilon,t}^*v\Psi_{\epsilon,t}^*\varphi_{x_0} R_{\Psi_{\epsilon,t}^*g_\epsilon(t)}d\mu_{\Psi_{\epsilon,t}^*g_\epsilon(t)}\notag\\
&\leq\int_M\left(\Delta_{\Psi_{\epsilon,t}^*g_\epsilon(t)} \Psi_{\epsilon,t}^*v\Psi_{\epsilon,t}^*\varphi_{x_0}+\Psi_{\epsilon,t}^*v\frac{\partial}{\partial t}\Psi_{\epsilon,t}^*\varphi_{x_0}    -\Psi_{\epsilon,t}^*v\Psi_{\epsilon,t}^*\varphi_{x_0} R_{\Psi_{\epsilon,t}^*g_\epsilon(t)}\right)d\mu_{\Psi_{\epsilon,t}^*g_\epsilon(t)}\notag\\
&=\int_M\left(\Psi_{\epsilon,t}^*v\Delta_{\Psi_{\epsilon,t}^*g_\epsilon(t)} \Psi_{\epsilon,t}^*\varphi_{x_0}+\Psi_{\epsilon,t}^*v\frac{\partial}{\partial t}\Psi_{\epsilon,t}^*\varphi_{x_0} -\Psi_{\epsilon,t}^*v\Psi_{\epsilon,t}^*\varphi_{x_0} R_{\Psi_{\epsilon,t}^*g_\epsilon(t)}\right)d\mu_{\Psi_{\epsilon,t}^*g_\epsilon(t)},
\end{align}

For the first term in the right hand side, we have
\begin{align}\label{R2}
\int_M\Psi_{\epsilon,t}^*v\Delta_{\Psi_{\epsilon,t}^*g_\epsilon(t)} \Psi_{\epsilon,t}^*\varphi_{x_0} d\mu_{\Psi_{\epsilon,t}^*g_\epsilon(t)}
=& \int_{B^h_1(x_0)}v\Delta_{g_\epsilon(t)} \varphi_{x_0} d\mu_{g_\epsilon(t)}\notag\\
\le& \frac{C(n,h,A,p)}{t^{n/2p}}\int_{B^h_1(x_0)}v d\mu_{g_\epsilon(t)}\notag\\
\le& \frac{C(n,h,A,p)}{t^{n/2p}}\int_{B^h_{C_2(n,h,A,p)}(\Psi_{\epsilon,t}^{-1}(x_0))}\Psi_{\epsilon,t}^*v d\mu_{\Psi_{\epsilon,t}^*g_\epsilon(t)}
\end{align}

For the second term  in the right hand side, by definition we have $\frac{\partial}{\partial t}\Psi_{\epsilon,t}^*\varphi_{x_0}(y)=W_\epsilon(t)|_{\Psi_{\epsilon,t}(y)}(\varphi_{x_0}(\Psi_{\epsilon,t}(y))$, where $W_\epsilon(t)$ is the family of vector fields in the definition of $h$-flow of initial metric $\hat g_\epsilon$. By Lemma \ref{thm6.2} we have $|W_\epsilon(t)|\le C(n,h,A,p)/t^{n/2p}$. Recall $|\tilde\nabla \varphi_x|\le C(n,h) $, thus we have
\begin{align}\label{R3}
\int_M \Psi_{\epsilon,t}^*v\frac{\partial}{\partial t}\Psi_{\epsilon,t}^*\varphi_{x_0}  d\mu_{\Psi_{\epsilon,t}^*g_\epsilon(t)}\le \frac{C(n,h,A,p)}{t^{n/2p}}\int_{B^h_{C_2(n,h,A,p)}(\Psi_{\epsilon,t}^{-1}(x_0))}\Psi_{\epsilon,t}^*v d\mu_{\Psi_{\epsilon,t}^*g_\epsilon(t)}.
\end{align}

For the third term  in the right hand side, by Lemma \ref{thm6.2} we have $|R_{\Psi_{\epsilon,t}^*g_\epsilon(t)}|=|R_{g_\epsilon(t)}\circ\Psi_{\epsilon,t}|\le C(n,h,A,p)/t^{\frac{n}{4p}+\frac{3}{4}}$. Thus we have
\begin{align}\label{R4}
|\int_M -\Psi_{\epsilon,t}^*v\Psi_{\epsilon,t}^*\varphi_{x_0} R_{\Psi_{\epsilon,t}^*g_\epsilon(t)}d\mu_{\Psi_{\epsilon,t}^*g_\epsilon(t)}|\le \frac{C(n,h,A,p)}{t^{\frac{n}{4p}+\frac{3}{4}}}\int_{B^h_{C_2(n,h,A,p)}(\Psi_{\epsilon,t}^{-1}(x_0))}\Psi_{\epsilon,t}^*v d\mu_{\Psi_{\epsilon,t}^*g_\epsilon(t)}.,
\end{align}

Combine \eqref{R1}, \eqref{R2}, \eqref{R3} and \eqref{R1}, we have

\begin{align}\label{R5}
\phi'_-(t)&\le \frac{C(n,h,A,p)}{t^q}\int_{B^h_{C_2(n,h,A,p)}(\Psi_{\epsilon,t}^{-1}(x_0))}\Psi_{\epsilon,t}^*v d\mu_{\Psi_{\epsilon,t}^*g_\epsilon(t)},
\end{align}
where $q={\text{max}}\{\frac{n}{2p},\frac{n}{4p}+\frac{3}{4}\}<1$.

We can choose $N$ balls $B^h_{C_3}(x_i),i=1,...,N$ such that  $B^h_{C_2(n,h,A,p)}(\Psi_{\epsilon,t}^{-1}(x_0))\subset \cup_{i=1}^N B^h_{C_3}(x_i)$, where $N$ only depends on $n$ and $h$. Recall that $\Psi_{\epsilon,t}^*\varphi_{\Psi_{\epsilon,t}(x_i)}=1$ on $B^h_{C_3}(x_i)$, thus we have
\begin{align}
\int_{B^h_{C_3}(x_i)}\Psi_{\epsilon,t}^*v d\mu_{\Psi_{\epsilon,t}^*g_\epsilon(t)}\le\int_M\Psi_{\epsilon,t}^*v\Psi_{\epsilon,t}^*\varphi_{\Psi_{\epsilon,t}(x_i)} d\mu_{\Psi_{\epsilon,t}^*g_\epsilon(t)}\le  \phi(t),
\end{align}

thus by \eqref{R5} we have
\begin{align}\label{ieqz}
\phi'_-(t)\leq \frac{C(n,h,A,p)}{t^q}\phi(t).
\end{align}

Denote \[\mathcal{T}=\{t_1\big|\forall\tau\in(t_1,t], \text{the supremum in the definition of $\phi(\tau)$ is attained in} \text{ }K_{\epsilon,\tau}\},\]and denote $t_0=\inf \mathcal{T}$. Then  \eqref{ieqz} holds for any $\tau\in(t_0,t]$.

Since $\phi$ is Lipschitz, $\phi'_-(\tau)=\phi'(\tau)$ almost everywhere.
we can multiple $e^{-\int_{t_0}^\tau\frac{C(n,h,A,p)}{s^q}ds}$ to both sides of \eqref{ieqz}, and get
\[[e^{-\int_{t_0}^\tau\frac{C(n,h,A,p)}{s^q}ds}\phi(\tau)]'\leq 0,\text{almost everywhere on $[{t_0},t]$}.\]

Integrate it on $[t_0,t]$, since $q=q(n,p)<1$ we get
\[\phi(t)\leq C(n,h,A,p)\phi(t_0)\]

Note that by the definition of $t_0$, if $t_0\ne0$, then we have $\phi(t_0)\leq 2\delta$, then $\phi(t)\leq C(n,h,A,p) (2\delta)$ and the claim holds. Otherwise, if $t_0=0$, then $\phi(t)\leq C(n,h,A,p)\phi(0)$ and the claim holds too.

Thus the claim holds for any case, and then we have
\begin{align}
\int_{B^h_{1/2}(x)}vd\mu_{g_\epsilon(t)}\le \phi(t)\le C(n,h,A,p)\sup_{x\in M}\int_{B^h_{1/2}(x)}v(0)d\mu_{\hat g_\epsilon}+C(n,h,A,p)\delta.
\end{align}

Note that $v\le (R_{g_\epsilon(t)})_-+\delta$, thus we have
\begin{align}
\int_{B^h_{1/2}(x)}vd\mu_{g_\epsilon(t)} \le C(n,h,A,p) \int_M (R_{\hat g_\epsilon})_-d\mu_{\hat g_\epsilon}+C(n,h,A,p)\delta.
\end{align}
Since for any fixed $(x,t)\in M\times[0,T]$, $\int_{B^h_{1/2}(x)}vd\mu_{g_\epsilon(t)}$ is continuous for $\delta$, we have $\phi(t)$ is lower semi-continuous for $\delta$. Thus
by letting $\delta\to 0^+$, we have
\begin{align}\label{tl}
\sup_{x\in M}\int_{B^h_{1/2}(x)}vd\mu_{g_\epsilon(t)} 
\le C(n,h,A,p)\int_M  \left|(R_{\hat g_\epsilon})_-\right|d\mu_{\hat g_\epsilon},\forall t\in(0,T],\forall x\in M
\end{align}

By Lemma \ref{thm3.1},
 \[\lowlim_{\epsilon\to0^+}\int_M (R_{\hat g_\epsilon})_-d\mu_h= 0,\]

 Thus taking limit in \eqref{tl}, we have
 \begin{align}\label{lmtR}
\lowlim_{\epsilon\to0^+}\int_{B^h_{1/2}(x)} \left|(R_{g_\epsilon(t)})_-\right|d\mu_{g_\epsilon(t)}=0,\forall t\in(0,T],\forall x\in M,
\end{align}
which confirms the lemma.
\end{proof}

Since for any fixed $t\in(0,T]$, $g_\epsilon(t)$ converges to $g(t)$ in $C^\infty$ norm, taking limit in \eqref{lmtR}, we have Lemma \ref{NNSC} holds.

\subsection{Proof of Theorem \ref{PMT}}
In this subsection we prove Theorem \ref{PMT}, we restate it here with a minor difference in the notation:

\begin{theorem}
Let $M^n(n\ge3)$ be a smooth manifold with a metric $\hat g\in C^\infty(M\setminus S)$, where $S$ is a subset of $M$ and $S=\cup_{k=0}^{k_0} S^k$, each $S^k(0\le k\le {k_0})$ is compact and $S^k\cap S^l=\emptyset$ for $k\neq l$, such that 
\begin{enumerate}
\item $S^0$ is a hypersurface which bounds a domain $\Omega$, $\hat g$ is Lipschitz in a neighborhood of $S^0$, and the mean curvature of $S^0$ in $(\bar\Omega,\hat g|_\Omega)$ is not less than  it in $(M\setminus \Omega,\hat g|_{M\setminus \Omega})$.

\item For each $1\le \iota\le m$,  $\hat g$ is $W^{1,p_\iota}$ in a neighborhood of $S^\iota$ for some $p_\iota \in (n,+\infty]$, and the $(n-1-n/{p_\iota})$-dimensional lower Minkowski content of $S^\iota$ is zero.
\end{enumerate}

Suppose $\hat g$ is complete and asymptotically flat, then its ADM mass is nonnegative. Moreoever, if the mass is zero, then there exists a $C^{1,\alpha}$ diffeomorphism $\Phi:\mathbb{R}^n \to M$, such that $g_{\text{Euc}}=\Phi^* \hat g$, where $g_{\text{Euc}}$ is the standard Euclidean metric.
\end{theorem}

\begin{proof}
Let $\hat g_\epsilon$ be the mollification in Lemma \ref{thm3.1}, let $h$ be some $\hat g_{\epsilon_0}$ such that $h$ is $(1+\frac{\epsilon(n)}{2})$-fair to $\hat g$ and $\hat g_\epsilon$ for $\epsilon$ small enough, where $\epsilon(n)$ is the constant in Lemma \ref{thm6.2}, and let $g_\epsilon(t)$ and $g(t)$ be the $h$-flow of initial metric $\hat g_\epsilon$ and $\hat g$ in Lemma \ref{thm6.2} respectively.

By Lemma \ref{ST1}, we have $g_\epsilon(t)$ is also asymptotically flat and
\begin{align}\label{m1}
m(\hat g)=m(\hat g_\epsilon)=m(g_\epsilon(t),
\end{align}
 where $m(\hat g), m(\hat g_\epsilon)$ and $m(g_\epsilon(t))$ are the ADM mass of $\hat g$, $\hat g_\epsilon$ and $g_\epsilon(t)$.

 As the same as in \cite{ShiTam}, Lemma \ref{ST1} actually says for any fixed $t\in(0,T]$, $g_\epsilon(t)$ satisfies the condition of \cite[Theorem 14]{McSz}, and by \cite[Theorem 14]{McSz} we have $g(t)$ is also asymptotically flat and
 \begin{align}\label{m2}
\lowlim_{\epsilon\to 0^+}m(g_\epsilon(t)\ge m(g(t)).
\end{align}

By Lemma \ref{NNSC}, we have $g(t)$ has nonnegative scalar curvature, thus by the classical positive mass theorem, we have 
\begin{align}\label{m3}
m(g(t))\ge 0.
\end{align}

Combine \eqref{m1}, \eqref{m2} and \eqref{m3}, we have
\begin{align}
m(\hat g)\ge 0.
\end{align}

Moreover, assume $m(\hat g)=0$, then Combining \eqref{m1}, \eqref{m2} and \eqref{m3} we have 
\begin{align}
m(g(t))= 0,\forall t\in(0,T].
\end{align}

By the classical positive mass theorem, $(M,g(t))$ are the Euclidean space. Since $g(t)$ converges to $\hat g$ in $C^0$ norm, and thus  $(M,g(t),p)$ converges to $(M,\hat g,p)$ in the sense of pointed Gromov-Hausdorff convergence for any $p\in M$. Thus $(M,\hat g)$ is Euclidean as a metric space. 

Thus we have a bijection
\begin{align}
\Phi:M\to \mathbb{R}^n,
\end{align}

Up to now, the standard Gromov-Hausdorff convergence theory only tells that $\Phi$ is an isometry for metric spaces. 

However, take any coordinate $(U,\varphi)$ on $M$ such that $U\subset\subset M$ and denote $\mathcal{O}=\varphi(U)$. Since $\hat g\in W_{\text{loc}}^{1,p}$ for $p>n$, by Morrey embedding we have $(\varphi^{-1})^*\hat g$ is $C^\alpha$ on $\mathcal{O}$ for some $\alpha\in(0,1)$. 

Note that
\begin{align}
\Phi\circ \varphi^{-1}:\varphi(U)\to\Phi(U)
\end{align} 
is a distance-preserving homeomorphism from $(\varphi(U),(\varphi^{-1})^*\hat g)$ to $(\Phi(U),g_{\text{Euc}})$, where $(\varphi^{-1})^*\hat g$ is $C^\alpha$ on $\mathcal{O}$, $g_{\text{Euc}}$ is the standard Euclidean metric and it is smooth on $\Phi(U)$. 

Therefore, by \cite[Theorem 2.1]{Ta06}, we have $\Phi\circ \varphi^{-1}$ is $C^{1,\alpha}$ and $(\Phi\circ \varphi^{-1})^* g_{\text{Euc}}=(\varphi^{-1})^*\hat g$. Thus $\Phi$ is $C^{1,\alpha}$, and $\Phi^*g_{\text{Euc}}=\hat g$. Similarly $(\Phi)^{-1}$ is also $C^{1,\alpha}$, which completes the proof of the theorem.
\end{proof}

\section{The inequality part}

The inequality part of the Penrose inequality admitting various kinds of singular sets has been essentially proved in  \cite[Proposition 3.1]{StMi} (see also \cite{Mi}). Though they only presented a result for metrics with corners along a hypersurface, their proof also works in our case.

 In fact, in \cite{Mi} and  \cite[Proposition 3.1]{StMi}, their argument based on an approximation developed in \cite{Miao2002}. However, if we replace it by the approximation in our Section 2, then the argument in \cite{Mi} and  \cite[Proposition 3.1]{StMi} would give the inequality part of Theorem \ref{mthm}.

For readers' convenience, we give the proof here.
\begin{proof}[Proof of the inequality part of Theorem \ref{mthm}]
As the same as the proof of \cite[Proposition 3.1]{StMi}, reflect $(M,g)$ to  across its boundary $\Sigma_H$ to get an asymptotically flat manifold $(\bar M,\bar g)$, which has two ends and has a corner along $\Sigma_H$. Let $M'$ denote  the image of the reflection of $M$ in $\bar M$ and let $S'$ denote the image of the reflection of $S$ into $\bar M$. Then $\bar g$ is smooth on $\bar M\setminus\left(S \cup S'\cup \Sigma_H\right)$

 Since $\Sigma_H$ is minimal, $S\cup S'$ has $n-1$-dimensional lower Minkowski content zero, and $\Sigma_H\cap \left(S\cup S'\right)=\emptyset$,  using Lemma \ref{thm6.2} we can construct a $h$-flow $g(t)$ of initial metric $\bar g$, such that 
 \begin{enumerate}
\item $g(t)$ has nonnegative scalar curvature (by Lemma \ref{NNSC}),
\item $g(t)$ converges to $\bar g$ in $C^0$ norm  as $t\to 0^+$ (by Lemma \ref{thm6.2}),
\item $g(t)$ is asymptotically flat with $m(g(t))\le m(g)$ (by Lemma \ref{ST1} and \cite[Theorem 14]{McSz}).
 \end{enumerate}

 Moreover, since $(\bar M,\bar g)$ is isometry under the reflection, we can take $h$ to preserve this reflection symmetry, and thus the flow $g(t)$ also preserves this reflection symmetry. Since the reflection is taken with respect to $\Sigma_H$, we have that $\Sigma_H$ is totally geodesic in $(\bar M,g(t))$, and thus minimal in it.

As the same as the proof of \cite[Proposition 3.1]{StMi}, now we turn our attention to the restriction of $g(t)$ on $M$. Let $\mathcal{S}$ be the set of all closed hypersurface in $M$ which enclose $\Sigma_H$. then since $3\le n\le 7$ and $\Sigma_H$ is minimal in $(\bar M,g(t))$, there is an hypersurface $\Sigma_t$ in $\mathcal{S}$ such that the area of $\Sigma_t$ attains the infimum of the area of all the hypersurface in $\mathcal{S}$.  Then $\Sigma_t$ is minimal in $(\bar M,g(t))$, and by the classical (Riemannian) Penrose inquality, we have
\begin{align}\label{ine1}
m(g)\ge m(g(t))\ge \frac{1}{2}\left(\frac{|\Sigma_t|_{g(t)}}{\omega_{n-1}}\right)^\frac{n-2}{n-1},
\end{align}
where $|\cdot|_{g(t)}$ is the area taken with respect to $g(t)$.

On the one hand, by definition we have $|\Sigma_t|_{g(t)}\le |\Sigma_H|_{g(t)}$, thus we have
\begin{align}\label{ine2}
\uplim_{t\to 0^+}|\Sigma_t|_{g(t)}\le \lim_{t\to 0^+}|\Sigma_H|_{g(t)}=|\Sigma_H|_g,
\end{align}
where $|\cdot|_{g}$ is the area taken with respect to $g$.

On the other hand, by \eqref{ine2} we have $|\Sigma_t|_{g(t)}\le C(g)$, where $C$ is a positive constant which is independent of $t$. Since $g(t)$ converges to $\bar g$ in $C^0$ norm  as $t\to 0^+$, by dominant convergence theorem we have
\begin{align}
\lim_{t\to 0^+}\left(|\Sigma_t|_{g(t)}-|\Sigma_t|_g\right)=0,
\end{align}
which tells that 
\begin{align}\label{ine3'}
\lowlim_{t\to 0^+}|\Sigma_t|_{g(t)}=\lowlim_{t\to 0^+}|\Sigma_t|_g.
\end{align}

Since $\Sigma_H$ is outer-minimizing in $(M,g)$, we have $|\Sigma_t|_g\ge |\Sigma_H|_g$, and \eqref{ine3'} becomes

\begin{align}\label{ine3}
\lowlim_{t\to 0^+}|\Sigma_t|_{g(t)}\ge|\Sigma_H|_g. 
\end{align}

Thus by \eqref{ine2} and \eqref{ine3}, we have
\begin{align}\label{ine0}
\lim_{t\to 0^+}|\Sigma_t|_{g(t)}\ge|\Sigma_H|_g. 
\end{align}

Thus taking limit in \eqref{ine1}, we have
\begin{align}
m(g)\ge \frac{1}{2}\left(\frac{|\Sigma_H|_{g}}{\omega_{n-1}}\right)^\frac{n-2}{n-1},
\end{align}
which proves the inequality part of Theorem \ref{mthm}.
\end{proof}

\section{The rigidity part}

Similar with \cite{LuMi}, the key for the proof of the main theorem is to study the rigidity of Bray's mass-capacity inequality, \cite[Theorem 9]{Br} for our singular metrics. Once this rigidity is proved, the remaining of the proof of our main theorem has no difference with the proof in \cite{LuMi}. Before we claim the mass-capacity inequality, we introduce the concept of the capacity of a hypersuface in an asymptotically flat manifold.

\begin{definition}[capacity]
Let $(M,g)$ be as the assumption of Theorem \ref{mthm} and let $\Sigma_H$ denote its nonempty bundary. 
 As the same as \cite{LuMi}, since $g$ is Lipschitz, there exists a function $\varphi$ satisfying the equation:
\begin{align*}\left\{
\begin{array}{rl}
\Delta_g \varphi&=0, \quad\  {\text{in}}\ M,\\
\varphi&=0, \quad\  {\text{on}}\ \Sigma_H,\\
\varphi(x)&\to 1,\quad {\text{as}}\ x\to \infty.
\end{array} \right.
\end{align*} 

Since $g$ is asymptotically flat, it is well known that $\varphi$ has such an asymptotic expansion:
\begin{align}
\varphi(x)=1-\frac{\mathcal{E}(g)}{2|x|^{2-n}}+o(|x|^{2-n}),\quad {\text{as}}\ x\to \infty,
\end{align}
where $\mathcal{E}(g)$ is a positive constant which is usually known as the capacity of $\Sigma_H$ in $(M,g)$.
\end{definition}

Standard elliptic theory tells that $\varphi$ is smooth on $M\setminus S$ ($\varphi$ is smooth up to $\Sigma_H$), and $\varphi\in W_{\text{loc}}^{2,p}(M)$, for any $p>n$. Thus we have $\varphi\in C_{\text{loc}}^{1,\alpha}(M)$, for any $\alpha\in (0,1)$. 

Now we can state the Mass-capacity inequality:
\begin{lemma}[Mass-capacity inequality for metrics with small singular sets]\label{MCInq}
Let $M^n$ be a smooth manifold and $g$ be a Lipschitz metric on $M$ which is smooth away from a compact singular set $S\subset M\setminus \Sigma_H$. Suppose $g$ is asymptotically flat and has nonnegative scalar curvature away from $S$, $M$ has a nonempty boundary $\Sigma_H$ with zero mean curvature. If $S$ is of vanishing $(n-1)$-dimensional lower Minkowski content.  Then 
\begin{align}
m(g)\ge \frac{\mathcal{E}(g)}{2}.
\end{align}

Moreover, if $m(g)= \frac{\mathcal{E}(g)}{2}$, then there exists a $C^{1,\alpha}$ diffeomorphism $\Phi:\mathbb{M}_m \to M$, such that $g_m=\Phi^* g$, where $(\mathbb{M}_m,g_m)$ is the standard spatial Schwarzschild manifold with mass $m$.
\end{lemma}
\begin{proof}
We consider a conformal metric of the double of $(M,g)$ as in \cite{Br} and \cite{LuMi}. Concretly speaking, we reflect $(M,g)$ to  across its boundary $\Sigma_H$ to get a asymptotically flat manifold $(\bar M,\bar g)$, which has two ends and has a corner along $\Sigma_H$. Let $M'$ denote  the image of the reflection of $M$ in $\bar M$ and let $S'$ denote the image of the reflection of $S$ into $\bar M$. Then $\bar g$ is smooth on $\bar M\setminus\left(S \cup S'\cup \Sigma_H\right)$. We let $\bar \varphi$ be the odd extension of $\varphi$ to $\bar M$, ($\varphi(p)=-\varphi(p')$ if $p'$ is the reflection point of $p$). Then we consider such a conformal metric:
\begin{align}\label{transg}
\tilde g=\left(\frac{\bar \varphi+1}{2}\right)^\frac{4}{n-2}\bar g \quad {\text{on}}\ \bar M.
\end{align}

Since $(M,g)$ is smooth away from $S$, by the standard elliptic regularity theory we have $\varphi$ is smooth away from $S$. Since $S\cap \Sigma_H=\emptyset$, $\varphi$ is smooth up to $\Sigma_H$ in $M$. Samely, $\bar \varphi$ is smooth up to $\Sigma_H$ in the side of $M'$. Since $\Sigma_H$ is a minimal surface in $(M,g)$ and $\bar \varphi$ is the odd extension, we have that the mean curvature of $\Sigma_H$ in $(M,\tilde g|_M)$ is equal to it in $(M',\tilde g|_{M'})$.

By the asymptotic expansion of $\varphi$, we know that $\tilde g$ is asymptotically flat with one end. In fact, it follows by a direct calculus that the end in $M$ remains an asymptotically flat end after this conformal transformation, while its mass is changed by the formula $m(\tilde g)=m(g)-\frac{\mathcal{E}(g)}{2}$. And the end in $M'$ is conformal to a punctured ball. Moreover, if we add a point $o$ representing the infinity of $M'$, then $\tilde g$ can be $W_{\text{loc}}^{1,q}$-extended to $\bar M\cup\{o\}$ for some $q>n$. (For a proof of this one-point conformal compactification, see \cite[Lemma 6.1]{MMT18}, or \cite[Lemma 4.3]{HM20}.)

In summary, $\tilde g\in C^2\left(\left(\bar M\cup\{o\}\right)\setminus\left(S \cup S'\cup \Sigma_H\cup\{o\}\right)\right)$, $\tilde g$ is Lipschitz on a neighborhood of $S\cup S'\cup \Sigma_H$ and is $W^{1,q}$ on a neighborhood of $o$, where the $(n-1)$-dimensional lower Minkowski content of $S\cup S'$ is zero and the $(n-1-\frac{n}{q})$-dimensional lower Minkowski content of $\{o\}$ is zero. $\tilde g$ has nonnegative scalar curvature on $\bar M\setminus\left(S \cup S'\cup \Sigma_H\right))$, and satisfies the mean curvature condition on $\Sigma_H$. 

Thus by Theorem \ref{PMT}, we have $m(\tilde g)\ge 0$, thus 
\begin{align}
m(g)=m(\tilde g)+\frac{\mathcal{E}(g)}{2}\ge \frac{\mathcal{E}(g)}{2}.
\end{align}
which proves the inequality part. 

 Moreover, if $m(g)=\frac{\mathcal{E}(g)}{2}$, then $m(\tilde g)=0$ and by Theorem \ref{PMT}, there exists a $C^{1,\alpha}$ diffeomorphism $\bar \Phi:\mathbb{R}^n \to \bar M\cup\{o\}$, such that $g_{\text{Euc}}=\bar \Phi^* \tilde g$, where $g_{\text{Euc}}$ is the standard Euclidean metric.

Now we prove the following claims:
\begin{enumerate}
\item $\bar \varphi$ is weakly $\bar g$-harmonic on $\bar M$.
\item Let $\bar \psi=\frac{2}{\bar \varphi +1}$. Then $\bar \psi$ is weakly $\tilde g$-harmonic on $\bar M$.
\end{enumerate}
 For the claim (1), since $\bar g$ is Lipschitz, there exists a function $\bar \phi$ satisfying
\begin{align*}\left\{
\begin{array}{rl}
\Delta_{\bar g} \bar \phi&=0, \quad\  {\text{on}}\ \bar M,\\
\bar \phi(x)&\to 1,\quad {\text{as}}\ x\to \infty.\\
\bar \phi(x)&\to -1,\quad {\text{as}}\ x\to \infty',
\end{array} \right.
\end{align*} 
where $\infty$ and $\infty'$ is the infinity of $M$ and $M'$ respectively.

Let $\bar w(p):=\bar \phi(p)+\bar \phi(p')$, where $p'$ is the reflection of $p$. Since $\bar g$ is symmetric with respect to the reflection of $\Sigma_H$, it follows that $\bar w$ satisfies
\begin{align*}\left\{
\begin{array}{rl}
\Delta_{\bar g} \bar w&=0, \quad\  {\text{on}}\ \bar M,\\
\bar w(x)&\to 0,\quad {\text{as}}\ x\to \infty.\\
\bar w(x)&\to 0,\quad {\text{as}}\ x\to \infty',
\end{array} \right.
\end{align*} 

Using maximal principle, we know that $\bar w\equiv 0$ on $\bar M$. Thus $\bar \phi$ is an odd function on $\bar M$ with respect to the reflection of $\Sigma_H$. And for any $q\in \Sigma_H$, since $q=q'$, we have 
\[\bar \phi(q)=\frac{1}{2}(\bar \phi(q)+\bar \phi(q'))=0.\]

Therefore, let $\phi$ be the restriction of $\bar \phi$  on $M$, then $\phi$ satisfies
\begin{align*}\left\{
\begin{array}{rl}
\Delta_g \phi&=0, \quad\  {\text{in}}\ M,\\
\phi&=0, \quad\  {\text{on}}\ \Sigma_H,\\
\phi(x)&\to 1,\quad {\text{as}}\ x\to \infty.
\end{array} \right.,
\end{align*}
which is totally the same as $\varphi$. Thus $\varphi\equiv \phi$ on $M$. Since $\bar \varphi$ and $\bar \phi$ are the odd extension of $\varphi$ and $\phi$ respectively, we have $\bar \varphi\equiv \bar \phi$ on $\bar M$. Thus the claim (1) holds.

For the claim (2), since $\tilde g=(\bar \psi^{-1})^\frac{4}{n-2} \bar g$, we have $d\mu_{\tilde g}=(\bar \psi^{-1})^\frac{2n}{n-2}d\mu_{\bar g}$, and $\tilde g(\omega,\omega')=\bar \psi^\frac{4}{n-2} \bar g(\omega,\omega')$ for any differential form fields $\omega$ and $\omega'$. For any test function $f\in W_0^{1,2}(\bar M)$, 

\begin{align}\label{int11}
\int_{\bar M}\tilde g(d\bar \psi,df) d\mu_{\tilde g}
=\int_{\bar M}\bar \psi^\frac{4}{n-2} \bar g(d\bar \psi,df) (\bar \psi^{-1})^\frac{2n}{n-2}d\mu_{\bar g}.
\end{align}

Since $d\bar \psi=-\bar \psi^2d(\bar \psi)^{-1}=-\frac{1}{2}\bar \psi^2d\bar \varphi$, we have
\begin{align}\label{int12}
&\int_{\bar M}\bar \psi^\frac{4}{n-2} \bar g(d\bar \psi,df) (\bar \psi^{-1})^\frac{2n}{n-2}d\mu_{\bar g}\notag\\
=&\int_{\bar M}\bar \psi^\frac{4}{n-2}\bar  g(-\frac{1}{2}\bar \psi^2d\bar \varphi,df) (\bar \psi^{-1})^\frac{2n}{n-2}d\mu_{\bar g}\notag\\
=&-\frac{1}{2}\int_{\bar M}\bar g(d\bar \varphi,df) d\mu_{\bar g}
\end{align}

Note that $\bar \varphi$ is weakly $\bar g$-harmonic on $\bar M$ and thus $\int_{\bar M}\bar g(d\bar \varphi,df) d\mu_{\bar g}=0$ (note that Sobolev spaces $W_0^{1,2}(\bar M)$ does not depend on the choice of the  metrics $\bar g$ or $\tilde g$). By \eqref{int11} and \eqref{int12} we have
\begin{align}
\int_{\bar M}\tilde g(d\bar \psi,df) d\mu_{\tilde g}=0,\forall f\in W_0^{1,2}(M).
\end{align}
Thus we have proved the claim (2).

Now we are able to solve $\bar \psi$ (under the pull back). Since $\Phi$ is a $C^{1,\alpha}$ diffeomorphism, by the definition of integral.
\begin{align}
0=\int_{\bar M}\tilde g(d \tilde \psi,d f)d\mu_{\tilde g}=\int_{\mathbb{R}^n\setminus \{0\}}g_{\text{Euc}}\left(d(\bar \psi\circ \bar \Phi),d(f\circ \bar \Phi)\right)d\mu_{g_{\text{Euc}}}.
\end{align}

Note that when $f$ varies in $W_0^{1,2}(\bar M,\tilde g)$, $f\circ \bar \Phi$ varies in $W_0^{1,2}(\mathbb {R}^n\setminus \{0\},g_{\text{Euc}})$. (This is because the composition of a $C^{1}$ diffeomorphism does not change whether a function belongs to $W^{1,2}$). Thus we have proved that $\bar \psi \circ \bar \Phi$ is weakly $g_{\text{Euc}}$-harmonic on $\mathbb {R}^n\setminus \{0\}$. Thus $\bar \psi \circ \bar \Phi$ is smooth on $\mathbb {R}^n\setminus \{0\}$. Therefore,  as the same as \cite{Br} and \cite{LuMi}, we have
\begin{align}\label{e1}
\bar \psi\circ \bar \Phi=1+\frac{m}{2|x|^{n-2}}.
\end{align} 

Applying the push back of $\bar\Phi$ to \eqref{transg}, since $\bar \psi=\frac{2}{\bar \varphi +1}$, we have
\begin{align}\label{e2}
g_{\text{Euc}}=\bar \Phi^*\tilde g=(\bar \psi \circ \bar \Phi)^\frac{-4}{n-2}\bar \Phi^*\bar g \quad {\text{on}}\ \bar M.
\end{align}
By \eqref{e1} and \eqref{e2}, we immediately get

\begin{align}
\bar \Phi^*\bar g=\left(1+\frac{m}{2|x|^{n-2}}\right)^\frac{4}{n-2}g_{\text{Euc}}=g_m \quad {\text{on}}\ \bar M.
\end{align}
Thus by restricting $\bar \Phi$ on the region outside the horizon, we have a $C^{1,\alpha}$ diffeomorphism  $\Phi:\mathbb{M}_m \to M$, such that $g_m=\Phi^* g$, where $(\mathbb{M}_m,g_m)$ is the standard spatial Schwarzschild manifold with mass $m$, which completes the proof of the lemma.
\end{proof}

Now Theorem \ref{mthm} follows from Lemma \ref{MCInq} as the same as \cite{LuMi} and \cite{Br}:

\begin{proof}[Proof of Theorem \ref{mthm}]
The inequality part has been proved in Section 3. And for the rigidity part, once our Lemma \ref{MCInq} is given, the proof is as the same as \cite{LuMi} and \cite{Br}, thus we only sketch it for the readers' convenience.

By the proof of the rigidity of Penrose inequality in \cite{LuMi} and \cite{Br}, there exists a family of Riemannian manifolds $(M_t,g_t)$, such that $g_t$ has the same type of  singular set as $(M,g)$, $\Sigma_H^{(t)}:=\partial M_t$ is outer-minimizing in $(M_t,g_t)$, and 
\begin{align}\label{pc}
m(g_t)\ge \frac{1}{2}\left(\frac{|\Sigma_H^{(t)}|_{g_t}}{\omega_{n-1}}\right)^\frac{n-2}{n-1},
\end{align}
where $|\Sigma_H^{(t)}|_{g_t}=|\Sigma_H|_g$. 

Moreover, we also have
\begin{align}
\frac{d}{dt}|_{t=0}m(g_t)=\mathcal{E}(g)-2m(g).
\end{align}

If $m(g)>\frac{\mathcal{E}(g)}{2}$, then $\frac{d}{dt}|_{t=0}m(g_t)<0$, which is a contradiction with \eqref{pc}, since we have assumed $m(g)=\frac{1}{2}\left(\frac{|\Sigma_H|_g}{\omega_{n-1}}\right)^\frac{n-2}{n-1}$.

Thus by Lemma \ref{MCInq}, we have $m(g)=\frac{\mathcal{E}(g)}{2}$ and there exists a $C^{1,\alpha}$ diffeomorphism $\Phi:\mathbb{M}_m \to M$, such that $g_m=\Phi^* g$, where $(\mathbb{M}_m,g_m)$ is the standard spatial Schwarzschild manifold with mass $m$, which completes the proof of the theorem.

\section*{Acknowledgements}
The author would like to thank his advisors Wenshuai Jiang and Weimin Sheng for their guidance and encouragement. He also  thanks Xi Zhang and Yuguang Shi for useful suggestions.
\end{proof}

\bibliographystyle{plain}

 \end{document}